\documentclass[reqno,10pt]{amsart}
\usepackage{amssymb}
\usepackage{amsmath}
\usepackage{amsthm}
\usepackage{mathrsfs}
\usepackage{pgf}
\usepackage{color}
\usepackage{graphicx}   
\usepackage{hyperref}

\usepackage{amsmath}
  \usepackage{paralist}
\usepackage{graphicx}  
\usepackage{psfrag}
\usepackage{comment}
\usepackage{esvect}

\newtheorem{theorem}{Theorem}[section]
\newtheorem{corollary}[theorem]{Corollary}

\newtheorem{lemma}[theorem]{Lemma}

\theoremstyle{definition}

\newtheorem{remark}[theorem]{Remark}

\newcommand{\R}{\mathbb{R}}
\newcommand{\N}{\mathbb{N}}

\newcommand{\eps}{\varepsilon}

\newcommand{\BBB}{\color{black}} 
\newcommand{\UUU}{\color{black}} 
 
\newcommand{\MMM}{\color{black}}  
 
\newcommand{\EEE}{\color{black}} 
\newcommand{\RRR}{\color{black}}

\usepackage{latexsym,amsfonts}
\usepackage{mathrsfs}
\usepackage{centernot}
\usepackage{paralist}

\usepackage{eso-pic}  
\usepackage{pifont}
\usepackage{fixmath}

\begin{document}

\title[A compactness result in $GSBV^p$]{A compactness result in $GSBV^p$ and applications to $\Gamma$-convergence for free discontinuity problems}

\keywords{variational fracture, free discontinuity problems, functions of bounded variation, piecewise Poincar\'e inequality,  $\Gamma$-convergence, homogenization}

\author{Manuel Friedrich}
\address[Manuel Friedrich]{Applied Mathematics M\"unster, University of M\"unster\\
Einsteinstrasse 62, 48149 M\"unster, Germany.}
\email{manuel.friedrich@uni-muenster.de}
\urladdr{https://www.uni-muenster.de/AMM/Friedrich/index.shtml}

\date{\today}

\begin{abstract}
We present a compactness result  in the space $GSBV^p$ which extends the classical statement due to Ambrosio \cite{Ambrosio:90} to \BBB problems \EEE without a priori bounds on the functions. As an application, we revisit the $\Gamma$-convergence results for free discontinuity \BBB functionals \EEE established recently by Cagnetti, Dal Maso, Scardia, and Zeppieri \cite{Caterina}. We investigate  sequences of boundary value problems and show   convergence of minimum values and minimizers.
\end{abstract}

\subjclass[2010]{ 49J45, 49Q20, 70G75, 74Q05, 74R10.} 
\maketitle

\section{Introduction}

Since the pioneering work of {\sc Griffith} \cite{Griffith:1921}, the propagation of crack is viewed as the result of a competition between elastic energy stored in the uncracked region of a body and dissipation related to an infinitesimal increase of the crack. It is the fundamental idea in the approach to quasistatic crack evolution by {\sc Francfort and Marigo} \cite{Francfort-Marigo:1998}  and has led to a variety of variational models, where the displacements and the 
  (a priori unknown) crack paths are determined from an energy minimization principle. (Among the vast body of literature, we mention here only  the brittle fracture models for small strains \cite{BJ, Chambolle:2003, Francfort-Larsen:2003, Solombrino, Giacomini-Ponsiglione:2006} and finite strains \cite{DalMaso-Francfort-Toader:2005, DalGiac, DalMaso-Lazzaroni:2010}, and the cohesive models  \cite{CagnettiToader, CrismaleLazzaroniOrlando, DalMasoZanini}.) Problems of this form may be formulated in the frame of free discontinuity functionals
\begin{align}\label{mainfunctional}
E(u) = \int_{\Omega} f(x,\nabla u(x)) \, dx + \int_{J_u} g(x,[u](x),\nu_u(x)) \, d\mathcal{H}^{d-1}(x). 
\end{align}
Here, $\Omega \subset \R^d$ denotes the reference configuration, $\nabla u$ the deformation gradient, and $J_u$ the crack surface. The energy density $f$ accounts for elastic bulk terms for the unfractured region of the body, whereas the surface term assigns energy contributions on the crack paths comparable to the $(d-1)$-dimensional Hausdorff measure $\mathcal{H}^{d-1}(J_u)$ of the crack. 

In its simplest formulation, the density $g$ is a constant, called \emph{toughness} of the material, which is given by Griffith's criterion of fracture initiation (see \cite{Griffith:1921}).  Densities $g$ depending explicitly on the crack opening $[u]$ allow for modeling fracture  problems of cohesive-type \cite{barenblatt}.  Finally, the presence of the normal $\nu_u$ to the jump set $J_u$ and the material point $x$ take into account possible anisotropy and inhomogeneities in the body.

A basic and important question   is to prove the existence of minimizers for \eqref{mainfunctional} under appropriate Dirichlet boundary conditions. This requires a \emph{weak formulation} of the problem in the space of \emph{special functions of bounded variation} ($SBV$) (see \cite[Section 4]{Ambrosio-Fusco-Pallara:2000}).  In \cite{Ambrosio:90, Ambrosio:90-2}, lower semicontinuity for functionals of the form \eqref{mainfunctional} is characterized in terms of quasiconvexity for $f$ and $BV$-ellipticity  \cite{AmbrosioBraides2} for $g$.  Compactness of sequences  with bounded energy  is guaranteed by an \emph{a priori}  bound on the functions in $L^\infty$, see   \cite{Ambrosio89, Ambrosio-Fusco-Pallara:2000}.

The drawback of this compactness result is that it is unfortunately difficult to obtain such uniform bounds for a minimizing sequence, even if lower order terms are present in the energy. Only in the antiplane case \cite{Francfort-Larsen:2003} (namely when the displacement $u$ is scalar and $f$ is of the form $f(x,\xi) = |\xi|^2$), $L^\infty$-bounds may be obtained by  truncation, assuming that also the prescribed boundary values are bounded in $L^\infty$. If the boundary datum is only in some $L^p$ space or $f(x,0) > \min_{\xi} f(x,\xi)$, which is typically the case in finite elasticity, \BBB a truncation may change the boundary values or increase the energy. \EEE   

This issue may be partially overcome by formulating the problem in the larger space of \emph{generalized special functions of bounded variation} ($GSBV$). In this setting, one can rely on the compactness result for $GSBV$ with respect to convergence in measure (see \cite{Ambrosio:90, Ambrosio-Fusco-Pallara:2000}): it requires only a   very mild control on the functions  of the form $\int_\Omega \psi(|u|)\, dx \le C$ for some nonnegative \UUU and \EEE continuous $\psi$ with $\lim_{t \to \infty}\psi(t) = +\infty$. \UUU Adding  a lower order \emph{fidelity term} of this kind \EEE to the energy,  compactness and eventually the existence of minimizers are  guaranteed.

Let us mention that similar compactness issues arise when dealing with a sequence of free discontinuity problems $(E_k)_k$ of the form  \eqref{mainfunctional}. A classical example for this situation is the case of periodic homogenization. Here, the densities are    of the form  $f_k(x,\xi) = f(x/\eps_k,\xi)$ and   $g_k(x,\zeta,\nu) = g(x/\eps_k,\zeta,\nu)$, where $f,g$ are periodic in the first variable and $\eps_k$ describes the microscopical scale of a microstructure. The effective asymptotic behavior for such a sequence of fracture models in the finite strain framework   was studied by {\sc Braides, Defranceschi, and Vitali} \cite{BDV}  by means of $\Gamma$-convergence \cite{Braides:02, DalMaso:93}. In particular, they show   convergence of minimum values and minimizers for boundary value problems under an \textit{a priori} $L^\infty$-bound on the deformations. 

Very recently, a generalization of these results  for  sequences of densities $f_k$ and $g_k$ without any periodicity assumptions and under more general growth conditions has been derived by {\sc Cagnetti, Dal Maso, Scardia, and Zeppieri} \cite{Caterina}. (Actually, \UUU their work is  motivated by studying the case of \EEE stochastic homogenization \cite{CDMSZ2}.) Here,  besides the size of a microstructure, the parameter $k$ may also have other interpretations, such as the scale of a regularization of the energy or the ratio of the contrasting value of the mechanical response in a high-contrast medium. The convergence of minimizers is shown by including an $L^p$-fidelity term $\Vert u - h\Vert_{L^p(\Omega)}$ in the energy for a suitable datum $h$.

 We emphasize that, in contrast to the case of image reconstruction, a fidelity term
is  in general not appropriate in fracture mechanics. An investigation of the problem \eqref{mainfunctional} only involving boundary conditions, without a priori bounds on the \BBB configurations \EEE or applied body forces, is desirable and in accordance with the original formulation of the problem \cite[Section 2]{Francfort-Marigo:1998}.  The main difficulty lies in the fact that, for configurations with finite energy \eqref{mainfunctional}, small pieces of the body could be completely disconnected from the bulk part by the jump set $J_u$ and  the function $u$ could take arbitrarily large values on such small components. Eventually, this may \BBB   rule out \EEE measure convergence for minimizing sequences. It seems that only including a fidelity term in the energy can exclude such a phenomenon.

The issue of compactness results in variational fracture was recently tackled from a slightly different direction, namely via models in linearized elasticity. They are formulated in the space of \emph{generalized special functions of bounded deformation} ($GSBD$) introduced by {\sc Dal Maso} \cite{DalMaso:13}. Although in this setting only the symmetric part $e(u) = \frac{1}{2}((\nabla u)^T + \nabla u)$ of the strain is controlled, similar compactness results under a priori $L^\infty$-bounds or mild fidelity terms have been established in \cite{Bellettini-Coscia-DalMaso:98} and \cite{DalMaso:13}, respectively. Nevertheless, the problem is   more severe with respect to the $SBV$-case since truncation methods are not applicable and thus  already the simple situation $f(x,\nabla u) = |e(u)|^2$ with boundary data in $L^\infty$ is a delicate problem. 

\UUU
The recent paper \cite{Solombrino} provides the first compactness and existence result in $GSBD$ for the Griffith energy in dimension two without any a priori bounds or fidelity terms. \EEE A related result \cite{Friedrich:15-2} has been obtained in the passage from nonlinear-to-linear energies in brittle fracture by means of $\Gamma$-convergence \BBB (see also \cite{Schmidt} for a discrete-to-continuum analysis). \EEE As discussed before, arbitrary minimizing sequences are typically not compact when (small) pieces are \BBB completely \EEE disconnected by the jump set. The compactness result relies on the idea that a \BBB control \EEE on a sequence of functions can always be ensured by subtracting suitable piecewise  rigid motions. Using a \emph{piecewise Korn inequality} \cite{Friedrich:15-4, Solombrino}, it can be shown that such a modification can be performed without essentially increasing the energy of the configurations.

Very recently, a related compactness result in $GSBD$ in arbitrary space dimensions has been derived  by {\sc Chambolle and Crismale} \cite{Crismale}. Their strategy relies on a Korn-Poincar\'e inequality for functions with small jump set \cite{Chambolle-Conti-Francfort:2014} together with arguments in the spirit of \BBB Rellich's type \EEE compactness theorems. In contrast to \cite{Solombrino}, no passage to modifications of a sequence $(u_k)_k$ is necessary, at the expense \BBB of the fact \EEE that convergence to a limiting function $u$ is only guaranteed \emph{outside} $A:=\lbrace x \in \Omega: \, |u_k(x)| \to \infty \rbrace$. On the one hand, by setting $u= 0$ on $A$ (or affine), this is enough to identify $u$ as a minimizer for certain fracture problems, including Griffith energies \cite{Iu3, Iu1, Solombrino} or approximations \emph{\`a la} Ambrosio-Tortorelli \cite{AmbrosioTortorelli, Crismale2}. On the other hand, this strategy is not expedient if ${\rm argmin}_{\xi} f(x,\xi)$ is $x$-dependent and therefore excludes a variety of interesting energies, e.g., models for composite materials. Moreover, this method is not adapted for applications to $\Gamma$-convergence where in general  sequences are supposed to converge on the \emph{whole} domain to a limiting function.

The main goal of the present paper is to derive a compactness result in the space $GSBV^p$, $p \in (1,\infty)$, \UUU without any a priori bounds or fidelity terms, \EEE  see Theorem \ref{th: comp}. We show that for a sequence of energies $(E_k)_k$  \UUU of the form \eqref{mainfunctional}, \EEE and for functions  $(u_k)_k  \subset GSBV^p(\Omega;\R^m)$ with  $\sup_{k \in \N} E_k(u_k)  <+\infty$ (possibly satisfying boundary conditions), one can find a subsequence, modifications $(y_k)_k \subset GSBV^p(\Omega;\R^m)$ (with the same boundary data \BBB as $(u_k)_k$) \EEE satisfying
$$
 E_k(y_k) \le E_k(u_k) + \tfrac{1}{k},  \ \ \ \ \ \  \mathcal{L}^d \big(\lbrace \nabla y_k \neq \nabla u_k \rbrace \big) \le \tfrac{1}{k},$$
and a limiting function  $u \in GSBV^p(\Omega;\R^m)$ such that
\begin{align*}
(i) & \ \  y_k \to u \text{ in measure on $\Omega$}, \\ 
(ii) & \ \ \nabla y_k  \rightharpoonup \nabla u \text{ weakly in } L^p,\\
(iii) & \ \ \mathcal{H}^{d-1}(J_u) \le \liminf_{k \to \infty} \mathcal{H}^{d-1}(J_{y_k}).
\end{align*}
Properties (ii) and (iii) also hold for the original sequence $(u_k)_k$. As explained above, it is in general indispensable to pass to modifications $(y_k)_k$  to ensure property (i). \UUU The class of admissible energies is very general: \EEE we only require standard growth conditions in $GSBV^p$ together with a mild monotonicity condition on $g$ used in \cite{Caterina}. (For details, see assumptions ($f1$)-($f2$) and ($g1$)-($g4$) in Section \ref{sec: compactness}.)

As applications, we prove existence of minimizers for energies of the form \eqref{mainfunctional} under Dirichlet boundary data.  Moreover, we revisit the $\Gamma$-convergence result for free discontinuity problems established recently in \cite{Caterina}. We show convergence of minimum values and minimizers for a sequence of boundary value problems without any  \UUU fidelity term. \EEE

To prove the main compactness result, we follow the strategy devised in \cite{Friedrich:15-2, Solombrino}: \UUU given a sequence of functions, \EEE we pass to suitable modifications whose \UUU energies coincide with the original ones \EEE up to an error   of order $\theta$. Subsequently, we let $\theta \to 0$ and apply  carefully a diagonal sequence argument (see Section \ref{sec: main proof}).  In contrast to the $GSBD$ setting where piecewise rigid motions have to be subtracted, in the present context of $GSBV^p$ functions we  can work with \emph{piecewise translated configurations}. Accordingly, the piecewise Korn inequality \cite{Friedrich:15-4} is replaced by a suitable \emph{piecewise Poincar\'e inequality} (see Section \ref{sec: piecewise Poin}), which is based on a careful use of the coarea formula in $BV$ (see \cite[Theorem 3.40]{Ambrosio-Fusco-Pallara:2000}). Let us note that the coarea formula has been largely employed to approximate $BV$ functions by piecewise constant functions, particularly to prove lower  semicontinuity \cite{Ambrosio:90}  and $\Gamma$-convergence results \cite{BDV, Caterina} in $SBV$, as well as the existence of quasistatic evolutions \cite{DalMaso-Francfort-Toader:2005, Francfort-Larsen:2003, Giacomini-Ponsiglione:2006}. Compared to \cite{Solombrino},  the passage to modifications is more delicate  due to the more general energies which may  depend explicitly on the crack opening. \BBB At this point, \EEE we draw some ideas from truncation methods in \cite{Caterina} and  use a mild monotonicity assumption on $g$ (see Section \ref{sec: piecewise}).

One of the main motivations for the compactness result is an application to $\Gamma$-convergence for free discontinuity problems. We extend the analysis in \cite{Caterina} by deriving a version  of the $\Gamma$-convergence result including Dirichlet boundary data. To this end, we follow the strategy in \cite[Lemma 7.1]{Giacomini-Ponsiglione:2006}. This eventually allows us to prove the convergence of minima and minimizers along a sequence of boundary value problems. 

The paper is organized as follows. In Section \ref{sec: preli} we first fix the notation and recall some basic properties. Section \ref{sec: compactness} contains the formulation of the main compactness result and its proof. In Section \ref{sec: appli} we finally provide two applications: an existence result for functionals of the form \eqref{mainfunctional} under Dirichlet boundary data and a convergence result  \BBB for \EEE a  sequence of functionals by means of $\Gamma$-convergence.


\section{Notation and preliminaries}\label{sec: preli}

In this section  we fix the notation and recall some basic tools.  

\textbf{Basic notation:} We use the  notations $\R^m_0 = \R^m \setminus \lbrace 0 \rbrace$,  $\mathbb{S}^{d-1} = \lbrace v \in \R^d: \ |v|=1\rbrace$, and $\R_+ = [0,+\infty)$.  For $\Omega \subset \R^d$ open and bounded, we denote by $\mathcal{A}(\Omega)$ the open subsets of $\Omega$. We use the symbol $\triangle$ for the symmetric difference of two sets in $\R^d$.  $\mathcal{L}^d$ denotes the Lebesgue measure on $\R^d$ and $\mathcal{H}^{d-1}$ the $(d-1)$-dimensional Hausdorff measure.   \BBB By \EEE  $L^0(\Omega;\R^m)$ we indicate the space of $\mathcal{L}^d$-measurable functions $u: \Omega \to \R^m$, endowed with the topology of convergence in measure. We observe that this convergence is metrizable. For $x \in \R^d$ and $\rho >0$ we denote by $B_\rho(x)$ the open ball with center $x$ and radius $\rho$. \UUU We denote the indicator function of $E \subset \Omega$ by $\chi_E$. \EEE

We will   use   the following measure-theoretical result. (See \cite[Lemma 4.1, 4.2]{Friedrich:15-2} and note that the \RRR statement in fact holds in arbitrary space dimensions.)  \EEE
\begin{lemma}\label{rig-lemma: concave function2}
Let $\Omega\subset \R^d$ with $\mathcal L^d(\Omega)<\infty$. Then for every sequence $(u_n)_n \subset L^1(\Omega;\R^m)$ with
$$\mathcal L^d\left(\bigcap\nolimits_{n \in \N} \bigcup\nolimits_{m \ge n} \lbrace |u_m - u_n| > 1 \rbrace\right)=0$$  
there exist a  subsequence (not relabeled) and an increasing concave function \BBB $\psi: \R_+ \to \R_+$ \EEE with 
$
\lim_{t\to \infty}\psi(t)=+\infty
$
such that   $$\sup_{n \ge 1} \int_{\Omega}\psi(|u_n|)\, dx < + \infty.$$ 
\end{lemma}

\textbf{$BV$ functions:}  For the general notions on $SBV$ and $GSBV$ functions and their  properties we refer to \cite{Ambrosio-Fusco-Pallara:2000}. For $u \in GSBV^p(\Omega;\R^m)$, $\Omega \subset \R^d$ open, we denote by $\nabla u$ the density of the absolutely
continuous part of $Du$ with respect to the Lebesgue measure $\mathcal{L}^d$. $J_u$ stands for the set of approximate jump points of $u$ and $\nu_u$ denotes the measure-theoretic normal to $J_u$. The symbols $u^\pm$ denote the one-sided approximate
limits of $u$ at a point of $J_u$ and we write $[u] = u^+ - u^-$.  We will also use the notation 
\begin{align}\label{eq: GSBVM}
GSBV_M^p(\Omega;\R^m) = \lbrace u \in GSBV^p(\Omega;\R^m): \ \Vert \nabla u \Vert_{L^p(\Omega)}^p + \mathcal{H}^{d-1}(J_u) \le M \rbrace.
\end{align} 
The following compactness result in $GSBV^p$ due to {\sc Ambrosio} \cite{Ambrosio:90} will be a key ingredient for our result.

\begin{theorem}\label{th: GSBD comp}
Let $\Omega \subset \R^d$ be open, bounded. Let $(u_k)_k$ be a sequence in $GSBV^p(\Omega;\R^m)$. Suppose that there exists a   continuous function $\psi: \R_+ \to \R_+$ with $\lim_{t \to \infty} \psi(t) = + \infty$ such that 
$$\sup_{k \in \N}  \Big(\int_{\Omega} \psi(|u_k|)\, dx + \int_{\Omega} |\nabla u_k|^p \, dx  + \mathcal{H}^{d-1}(J_{u_k}) \Big) < + \infty.$$
Then there exists a subsequence, still denoted by $(u_k)_k$, and a function $u \in GSBV^p(\Omega;\R^m)$ such that
$u_k \to u$ in measure on $\Omega$, $\nabla u_k \rightharpoonup \nabla u $ weakly in $L^p(\Omega;\R^{m \times d})$, and $\mathcal{H}^{d-1}(J_u) \le \liminf_{k \to \infty} \mathcal{H}^{d-1}(J_{u_k})$.
\end{theorem}

\textbf{Caccioppoli partitions:} We say that a partition $\mathcal{P} = (P_j)_j$ of an open set $\Omega\subset \R^d$ is a \textit{Caccioppoli partition} of $\Omega$ if $\sum\nolimits_j \mathcal{H}^{d-1}(\partial^* P_j) < + \infty$, where $\partial^* P_j$ denotes  the \emph{essential boundary} of $P_j$ (see \cite[Definition 3.60]{Ambrosio-Fusco-Pallara:2000}). We say a  partition is \textit{ordered} if $\mathcal L^d(P_i) \ge \mathcal L^d(P_j)$ for $i \le j$. The  local structure of Caccioppoli partitions can be characterized as follows (see \cite[Theorem 4.17]{Ambrosio-Fusco-Pallara:2000}).
 
\begin{theorem}\label{th: local structure}
Let $(P_j)_j$ be a Caccioppoli partition of $\Omega$. Then 
$$\bigcup\nolimits_j (P_j)^1 \cup \bigcup\nolimits_{i \neq j} (\partial^* P_i \cap \partial^* P_j)$$
contains $\mathcal{H}^{d-1}$-almost all of $\Omega$.
\end{theorem}
Here $(P)^1$ denote the points where $P$ has density one (see again \cite[Definition 3.60]{Ambrosio-Fusco-Pallara:2000}). 
 Essentially, the  theorem states that $\mathcal{H}^{d-1}$-a.e.\ point of $\Omega$ either belongs to exactly one element of the partition or to the intersection of exactly two sets $\partial^* P_i$, $\partial^* P_j$. We now state a compactness result for ordered Caccioppoli partitions. (See \cite[Theorem 4.19, Remark 4.20]{Ambrosio-Fusco-Pallara:2000} or \cite[Theorem 2.8]{Solombrino} for the slightly adapted version presented here.)  

\begin{theorem}\label{th: comp cacciop}
Let $\Omega \subset \R^d$ be a bounded Lipschitz domain.  Let $\mathcal{P}_i = (P_{j,i})_j$, $i \in \N$, be a sequence of ordered Caccioppoli partitions of $\Omega$ with $$\sup\nolimits_{i \ge 1} \sum\nolimits_{j}\mathcal{H}^{d-1}(\partial^* P_{j,i}) < + \infty.$$
Then there exists a Caccioppoli partition $\mathcal{P} = (P_j)_j$ \UUU of $\Omega$ \EEE and a  subsequence (not relabeled) such that  
$\sum_{j} \mathcal L^d\left(P_{j,i} \triangle  P_j\right) \to 0$  as $i \to \infty$.
\end{theorem}

The proof in  \cite{Ambrosio-Fusco-Pallara:2000} shows that the result still holds if the assumption of ordered partitions is replaced by the weaker assumption that for fixed $j_0 \in \N$ only $(P_{j,i})_{j\ge j_0}$ are ordered, i.e., $\mathcal L^d(P_{j,i}) \ge \mathcal L^d(P_{k,i})$ for all $j_0 \le j \le k$ and $i \in \N$.

The starting point for the construction of piecewise translated configurations will be the following \MMM approximation of $GSBV$ functions by piecewise constant functions, which can be seen as a \EEE piecewise Poincar\'e inequality.

 \begin{theorem}\label{th: main poinc}
Let $\Omega \subset \R^d$ be a bounded Lipschitz domain. Let $m \in \N$. Then there exists  a constant $C_0=C_0(m) \ge 1$   such that for  each $u \in (GSBV(\Omega; \R))^m$ with $\Vert  \nabla u\Vert_{L^1(\Omega)} + \mathcal{H}^{d-1}(J_u) < + \infty$  there exists a Caccioppoli partition $(P_j)_{j=1}^\infty$ of $\Omega$ and corresponding translations $(b_j)_{j=1}^\infty \subset \R^m$ such that  $v:= u - \sum_{j=1}^\infty b_j \chi_{P_j} \in SBV(\Omega;\R^m) \cap L^\infty(\Omega;\R^m)$ and 
\begin{align}\label{eq: kornpoinsharp0}
\begin{split}
(i) & \ \   \sum\nolimits_{j=1}^\infty \mathcal{H}^{d-1}( \partial^* P_j ) \le   2 \mathcal{H}^{d-1}(J_u\cup \partial \Omega) +1,\\
(ii) &\ \ \Vert v \Vert_{L^{\infty}(\Omega)} \le C_0  \Vert  \nabla u\Vert_{L^1(\Omega)}.
\end{split}
\end{align}
\end{theorem}

\MMM This result  essentially relies on the \emph{coarea formula} in $BV$ (see \cite[Theorem 3.40]{Ambrosio-Fusco-Pallara:2000}), where the sets $P_j$ are chosen as the intersection of suitable level sets of the components $u_i$, $i=1,\ldots,m$.  For the proof we refer to  \cite[Theorem 2.3]{Friedrich:15-4}, but we also mention that the argument can be found in previous literature, e.g., in  \cite[Theorem 3.3]{Ambrosio:90} and \cite[Proposition 6.2]{BDV}. \EEE

\section{Compactness result in $GSBV^p$}\label{sec: compactness}

In this section we formulate and prove the main compactness result. 

\subsection{Formulation of the main compactness result}\label{sec: formulation}

Throughout the paper we fix the constants $p \in (1,\infty)$, $0< c_1 \le c_2 < + \infty$, $1 \le c_3 < + \infty$, and $0< c_4 < c_5 < + \infty$. We will consider integral functionals with bulk densities $f: \Omega \times \R^{m \times d} \to \R_+$ satisfying the conditions
\begin{itemize}
\item[($f1$)]  (measurability) $f$ is Borel measurable on $\Omega \times \R^{m\times d}$,
\item[($f2$)] (lower and upper bound) for every $x \in \Omega$ and every $\xi \in \R^{m \times d}$ 
$$c_1 |\xi|^p \le f(x,\xi) \le c_2 (1+ |\xi|^p)$$
\end{itemize}
and surface densities $g: \Omega \times \R^m_0 \times \mathbb{S}^{d-1} \to \R_+$ satisfying the conditions  
\begin{itemize}
\item[($g1$)] (measurability) $g$ is Borel measurable on $\Omega \times \R^m_0 \times \mathbb{S}^{d-1}$, 
\item[($g2$)] (estimate for $c_3 |\zeta_1| \le |\zeta_2|$) for every $x \in \Omega$ and every $\nu \in \mathbb{S}^{d-1}$ we have
$$g(x, \zeta_1 , \nu)   \le  g(x, \zeta_2, \nu)$$
for every $\zeta_1, \zeta_2 \in \R^m_0$ with $c_3|\zeta_1| \le |\zeta_2|$,
\item[($g3$)] (lower and upper bound) \RRR for every $x \in \Omega$, $\zeta \in \R^m_0$, and  $\nu \in \mathbb{S}^{d-1}$ we have \EEE
$$c_4  \le g(x,\zeta,\nu)  \le  c_5,$$
\item[($g4$)] (symmetry) for every $x \in \Omega$, $\zeta \in \R^m_0$, and  $\nu \in \mathbb{S}^{d-1}$ we have 
$$g(x,\zeta,\nu) = g(x,-\zeta,-\nu).$$
\end{itemize}
We let $\mathcal{E}_\Omega = \mathcal{E}_\Omega(\Omega,c_1,c_2,c_3,c_4,c_5,p)$ be the collection of   all integral functionals  $E: L^0(\Omega; \R^m) \times \mathcal{A}(\Omega)\to [0,+\infty]$ defined by
\begin{align}\label{eq: main energy functional}
E(u,A) = \begin{cases} \int_{A} f(x,\nabla u(x)) \, dx + \int_{J_u \cap A} g(x,[u](x), \nu_u(x)) \, d\mathcal{H}^{d-1}(x) & \text{if } u|_A \in GSBV^p(A;\R^m), \\
+ \infty & \text{else,} \end{cases}
\end{align}
where $f: \Omega \times \R^{m \times d} \to \R_+$ satisfies ($f1$)-($f2$) and $g: \Omega \times \R^m_0 \times \mathbb{S}^{d-1} \to \R_+$ satisfies ($g1$)-($g4$). \UUU (The dependence of $E$ on subsets of $\Omega$ will be convenient for our applications in Section \ref{sec: gamma}.) For simplicity, we \EEE write $E(u,\Omega) = E(u)$.

We remark that, apart from ($g2$), the assumptions on the bulk and surface densities are standard. In particular, the symmetry condition ($g4$) ensures that $E$ is well defined since $[u]$ is reversed if the orientation of $\nu_u$ is reversed.  Assumption ($g2$) was used in \cite{Caterina}. Among others, it includes the case of densities that are `monotonic' in the jump height $|\zeta|$, see \cite[Remark 3.2]{Caterina} for further details. In the proof of the main compactness result, this condition is necessary to pass to \emph{piecewise translated configurations} without essentially increasing the energy, see Section \ref{sec: piecewise} for details.

The following theorem is the main result of the paper. 

\begin{theorem}[Compactness in $GSBV^p$]\label{th: comp}
Let $\Omega \subset \Omega' \subset \R^d$ be bounded Lipschitz domains. Let $(E_k)_k \subset \mathcal{E}_{\Omega'}$ and let $(h_k)_k \subset W^{1,p}(\Omega';\R^m)$ converging in $L^p(\Omega';\R^m)$ to some $h \in W^{1,p}(\Omega';\R^m)$ such that $(|\nabla h_k|^p)_k$ are equi-integrable. Consider $(u_k)_k  \subset GSBV^p(\Omega';\R^m)$ with $u_k = h_k$ on $\Omega' \setminus \overline{\Omega}$ and  $\sup_{k \in \N} E_k(u_k)  <+\infty$.

\noindent Then we find a  subsequence (not relabeled), modifications $(y_k)_k \subset GSBV^p(\Omega';\R^m)$ satisfying
\begin{align}\label{eq: compi1}
\text{$y_k = h_k$ on $\Omega' \setminus \overline{\Omega}$,} \ \ \ \ \ \   E_k(y_k) \le E_k(u_k) + \tfrac{1}{k},  \ \ \ \ \ \  \mathcal{L}^d \big(\lbrace \nabla y_k \neq \nabla u_k \rbrace \big) \le \tfrac{1}{k},
\end{align}
and a limiting function  $u \in GSBV^p(\Omega';\R^m)$ with $u = h$ on $\Omega' \setminus \overline{\Omega}$ such that 
\begin{align*}
(i)& \ \ \text{$y_k \to u$ in measure on $\Omega'$,}\notag\\ 
(ii)& \ \ \text{$\nabla y_k  \rightharpoonup \nabla u$   weakly in $L^p(\Omega'; \R^{m\times d})$}.
\end{align*}
Moreover,  $\nabla u_k  \rightharpoonup \nabla u$   weakly in $L^p(\Omega'; \R^{m\times d})$, and
\begin{align}\label{eq: jump charac}
\mathcal{H}^{d-1}(J_u) \le \liminf_{k \to \infty} \mathcal{H}^{d-1}(J_{y_k}) \le  \liminf_{k \to \infty} \mathcal{H}^{d-1}(J_{u_k}).
\end{align}

\end{theorem}

We emphasize that in general it is indispensable to replace the functions $(u_k)_k$ by certain modifications $(y_k)_k$. Consider, e.g., the sequence $u_k = k \chi_U$ for some set $U \subset \Omega$ of finite perimeter. \RRR Then $E_k(u_k) \le c_5\mathcal{H}^{d-1}(\partial^* U) \UUU + c_2 \mathcal{L}^d(\Omega')\EEE$ by ($f2$) and ($g3$) which is uniformly controlled. \EEE However, $u_k$ does not converge \RRR in measure \EEE  on $U$. 

The idea in the proof is to construct $y_k$ from $u_k$ by subtracting  a function which is piecewise constant \BBB (up to a set of small measure). \EEE This prevents that the \BBB functions `escape to infinity' \EEE on subsets which are completely disconnected  from the rest of the domain by the jump set. The construction \UUU also implies \EEE  that $\nabla y_k$ coincides with $\nabla u_k$ outside \UUU of a  small set  whose measure vanishes \EEE for $k \to \infty$. Thus, $\nabla u_k  \rightharpoonup \nabla u$   weakly in $L^p$ also holds for the original sequence $(u_k)_k$. Moreover, by this construction the jump set is asymptotically not increased, see \eqref{eq: jump charac}.

The result is proved in the following three subsections. In Section \ref{sec: piecewise} we first construct  piecewise translated configurations $(v^\theta_k)_k$ which are bounded in $L^\infty$ \MMM by a constant $C_\theta$ depending on  $\theta$ with $C_\theta \to \infty$ as $\theta \to 0$. \EEE Their energies coincide with the ones  of $(u_k)_k$ \EEE up to a (small) error \BBB  of order \EEE $\theta$. This construction \MMM exploits the monotonicity assumption ($g2$)  and \EEE relies on a suitable piecewise Poincar\'e inequality which is proved in Section \ref{sec: piecewise Poin}. Finally, in Section \ref{sec: main proof} we define  the sequence $(y_k)_k$ by letting $\theta \to 0$ and choosing a diagonal sequence in $(v^\theta_k)_{k,\theta}$.  \MMM The choice of the latter is quite delicate since the $L^\infty$-control $C_\theta$ blows up for $\theta \to 0$. Additional arguments involving Lemma \ref{rig-lemma: concave function2} are necessary to show that we can apply Theorem \ref{th: GSBD comp} on  $(y_k)_k$.  \EEE

\subsection{Piecewise translated configurations}\label{sec: piecewise}

Recall the definition of $GSBV_M^p(\Omega;\R^m)$ in \eqref{eq: GSBVM}. The goal of this section is to prove the following result.

\begin{theorem}[Piecewise translated configurations]\label{thm: translations}
Let $\Omega \subset \R^d$ be a bounded Lipschitz domain. Let $M>0$ and $0 < \theta < 1$. Then there exist constants $C_M = C_M(M,\Omega,  \lbrace c_i \rbrace_i,p )>0$ and  $C_{\theta, M}=C_{\theta, M}(M,  \theta, \Omega, \lbrace c_i \rbrace_i,p )>0$ such that \BBB the following holds: \EEE for each $u \in GSBV^p_M(\Omega;\R^m)$ we find a finite Caccioppoli partition $\Omega = \bigcup_{j=1}^J P_j \cup R$ as well as translations $(t_j)_{j=1}^J$ such that $v:= \sum_{j=1}^J(u - t_j) \chi_{P_j} \in  SBV^p  (\Omega;\R^m)$ \BBB and we have \EEE
\begin{align}\label{eq: main estimates}
(i) & \ \ E(v) \le E(u)  + C_M \theta, \notag \\
(ii) & \ \ \mathcal{H}^{d-1}(J_v) \le \mathcal{H}^{d-1}(J_u) +  C_M\theta, \notag \\
(iii) &  \ \  \Vert v \Vert_{L^\infty(\Omega)} \le C_{\theta, M},\notag \\
(iv) & \ \ \mathcal{L}^d(R) \le C_M\theta, \notag\\
(v) & \ \ \sum\nolimits_{j = 1}^J \mathcal{H}^{d-1} (\partial^* P_j) + \mathcal{H}^{d-1} (\partial^* R) \le C_M
\end{align}
 for all energies $E \in \mathcal{E}_\Omega$. Moreover, we have $\lbrace v = 0\rbrace \supset \lbrace u = 0\rbrace$ \BBB (up to a set of negligible measure). \EEE  Finally, for each collection $(t'_j)_{j=1}^J$ with $|t_j - t_j'| \le \UUU \theta^{-1} \EEE \Vert v\Vert_{L^\infty(\Omega)}$ for $j=1,\ldots,J$, the function $v' := \sum_{j=1}^J(u - t'_j) \chi_{P_j}$ also satisfies 
\begin{align}\label{eq: extra energy}
E(v') \le E(u)  + C_M \theta \ \  \ \text{for all} \ E \in \mathcal{E}_\Omega, \ \ \  \ \ \ \  \mathcal{H}^{d-1}(J_{v'}) \le \mathcal{H}^{d-1}(J_u) +  C_M\theta.
\end{align}
\end{theorem}

Outside the \emph{rest set} $R$, $v$ arises from $u$ by subtracting a piecewise constant function. Therefore, we call $v$ a \emph{piecewise translated configuration}. The rest set is related to a piecewise Poincar\'e inequality, see Lemma  \ref{lemma: key} below \UUU and the comments thereafter. \EEE 

A similar result has been derived in \cite[Theorem 4.1]{Solombrino} for a two-dimensional Griffith model in $SBD$ where piecewise rigid motions are subtracted to obtain uniformly bounded functions. \MMM If the density $g$ in \eqref{eq: main energy functional} is constant (as in \cite{Solombrino}), property \eqref{eq: main estimates}(i) follows essentially from \eqref{eq: main estimates}(ii). If, however, $g$ depends explicitly on the jump height, the energy is in general affected by passing to piecewise translated configurations. In this case, the proof is much more delicate: the components  $(P_j)_{j=1}^J$ and the constants $(t_j)_{j=1}^J$ have to be chosen in a careful way, and one needs to use ($g2$) to ensure the energy estimate \eqref{eq: main estimates}(i). This is subject of Lemma \ref{lemma: key} below which is a refinement of Theorem \ref{th: main poinc}. \EEE In the proof we will combine the strategy in \cite{Solombrino} with ideas inspired by a truncation method for $GSBV$ functions \cite{Caterina}. 

We remark that truncations, as used in \cite{BDV, Caterina}, also yield a uniform bound of the form \eqref{eq: main estimates}(iii). In that case, however,   in the energy estimate \eqref{eq: main estimates}(i), an additional term $c_2 \mathcal{L}^d(\lbrace |u| \ge \lambda \rbrace)$ occurs, where $\lambda$ represents the level of truncation \UUU (see, e.g., \cite[Lemma 4.1]{Caterina}). \EEE Along a sequence $(u_k)_k$ from Theorem \ref{th: comp},  we cannot expect that $ \mathcal{L}^d(\lbrace|u_k| \ge \lambda \rbrace) \to 0$ as $k \to \infty$. Thus, truncations could perturb the energy significantly and are thus not expedient in the present context. 

We now formulate a version of Theorem \ref{thm: translations} for functions satisfying boundary conditions. 

\begin{corollary}[Piecewise translated configurations with boundary conditions]\label{cor: translations-bdy} 
Let $\Omega \subset \Omega' \subset \R^d$ be bounded Lipschitz domains. Let $M>0$, $0 < \theta < 1$. Then there exist constants $C_M = C_M(M, \Omega', \lbrace c_i \rbrace_i,p )>0$ and  $C_{\theta, M}=C_{\theta, M}(M,  \theta, \Omega', \lbrace c_i \rbrace_i,p )>0$ such that the following holds:  for each $h \in W^{1,p}(\Omega'; \R^m)$ \UUU with $\Vert \nabla h \Vert^p_{L^p(\Omega')} \le M$ \EEE and each $u \in GSBV^p_M(\Omega';\R^m)$ with $u = h$ on $\Omega' \setminus \overline{\Omega}$ we find a finite Caccioppoli partition $\Omega' = \bigcup_{j=1}^J P_j \cup R$ as well as translations $(t_j)_{j=1}^J$ such that $v:= h \chi_R + \sum_{j=1}^J(u - t_j) \chi_{P_j} \in SBV^p  (\Omega';\R^m)$
satisfies $v = h$ on $\Omega' \setminus \overline{\Omega}$ and
\begin{align}\label{eq: small set main-boundary}
(i) & \ \ E(v) \le E(u)  + C_M \theta   + C_M \Vert \nabla h \Vert_{L^p(R)}^p, \notag \\
(ii) & \ \ \mathcal{H}^{d-1}(J_v) \le \mathcal{H}^{d-1}(J_u) + C_M\theta, \notag \\
(iii) &  \ \  \Vert v -h\Vert_{L^\infty(\Omega')} \le C_{\theta, M},\notag \\
(iv) & \ \ \mathcal{L}^d(R) \le C_M\theta, \notag\\
(v) & \ \ \sum\nolimits_{j = 1}^J \mathcal{H}^{d-1} (\partial^* P_j) + \mathcal{H}^{d-1} (\partial^* R) \le C_M
\end{align}
 for all energies $E \in \mathcal{E}_{\Omega'}$. Moreover, for each collection $(t'_j)_{j=1}^J$ with $|t_j - t_j'| \le \UUU \theta^{-1} \EEE \Vert v-h\Vert_{L^\infty(\Omega')}$ for $j=1,\ldots,J$, the function $v' := h \chi_R +  \sum_{j=1}^J(u - t'_j) \chi_{P_j}$ also satisfies
\begin{align}\label{eq: extra energy2}
E(v') \le E(u)  + C_M \theta  + C_M \Vert \nabla h \Vert_{L^p(R)}^p, \ \ \ \ \ \ \  \mathcal{H}^{d-1}(J_{v'}) \le \mathcal{H}^{d-1}(J_u) + C_M\theta
\end{align}
for all $E \in \mathcal{E}_{\Omega'}$. Finally, there is at most one component $P_j$ intersecting $\Omega' \setminus \overline{\Omega}$.  
\end{corollary}

\BBB The idea in the proof is to apply Theorem \ref{thm: translations} on $u-h$. The property \EEE  that at most one component   intersects $\Omega' \setminus \overline{\Omega}$ can be seen as follows: for each $P_j$ intersecting   $\Omega' \setminus \overline{\Omega}$ we have $t_j = 0$ since $u=v = h$ on $\Omega' \setminus \overline{\Omega}$. Thus, if different components  intersected   $\Omega' \setminus \overline{\Omega}$, \BBB they could simply \EEE be combined to just one component. We again remark that truncations  \cite{BDV, Caterina}  can not be applied here since they in general do not preserve boundary conditions.

Corollary \ref{cor: translations-bdy} implies the following approximation result, which we will use in Section \ref{sec: gamma}. 

\begin{corollary}[Approximation by $L^p$ functions]\label{cor: truncations} 
Let $\Omega \subset \Omega' \subset \R^d$ be bounded Lipschitz domains. Let $h \in W^{1,p}(\Omega'; \R^m)$ and $E \in \mathcal{E}_{\Omega'}$. Then  for each $u \in GSBV^p(\Omega';\R^m)$ with $u = h$ on $\Omega' \setminus \overline{\Omega}$ we find a sequence $(u_k)_k \subset GSBV^p(\Omega';\R^m) \cap L^p(\Omega';\R^m)$ with $u_k = h$ on $\Omega' \setminus \overline{\Omega}$ such that $u_k \to u$ in measure on $\Omega'$ and   $\limsup_{k \to \infty} E(u_k)\le E(u)$.
\end{corollary}

A key ingredient for the proof of Theorem \ref{thm: translations}--Corollary \ref{cor: truncations} will be the following result, which is a refinement of the piecewise Poincar\'e inequality stated in Theorem \ref{th: main poinc}.

\begin{lemma}[Piecewise Poincar\'e inequality with additional control on translations]\label{lemma: key}
Let $\Omega \subset \R^d$ be a bounded Lipschitz domain. Let $\alpha \ge 1$ and  $0 < \theta < 1$. Then there exist  constants $C_\Omega = C_\Omega(\Omega) \UUU \ge 1\EEE$ and $C_{\theta,\alpha} = C_{\theta,\alpha}(\theta,\alpha)>0$ such that the following holds: for each $u \in GSBV^p(\Omega;\R^m)$ we find
 a finite Caccioppoli partition $\Omega = \bigcup_{j=1}^J P_j \cup R_1 \cup R_2$ with 
\begin{align}\label{eq: R1R2-new} 
(i) & \ \ \mathcal{L}^d(R_1 \cup R_2) \le C_\Omega\theta  \, \mathcal{H}^{d-1}(J_u \cup \partial \Omega), \notag \\
(ii) & \ \ \mathcal{H}^{d-1}(\partial^* R_1) \le C_\Omega\theta \, \mathcal{H}^{d-1}(J_u \cup \partial \Omega), \notag \\
(iii) & \ \ \sum\nolimits_{j=1}^J \mathcal{H}^{d-1}(\partial^* P_j) + \mathcal{H}^{d-1}(\partial^* R_2) \le C_\Omega\mathcal{H}^{d-1}(J_u \cup \partial \Omega) 
\end{align}
as well as translations $(b_j)_{j=1}^J$ and $\lambda_{\theta,\alpha} \in [ \UUU 1, \EEE C_{\theta,\alpha}]$ such that
\begin{align}\label{eq:II-new}
(i) &  \ \ \Vert u - b_j \Vert_{L^\infty(P_j)} \le \lambda_{\theta,\alpha} \Vert \nabla u \Vert_{L^1(\Omega)} \ \ \ \ \text{for} \ 1 \le j \le J,\notag\\
(ii) & \ \ \min\nolimits_{1 \le j \le J} \  {\rm ess \, inf} \lbrace  |u(x) -  b_j|: \, x \in R_2 \rbrace  \ge  \alpha \lambda_{\theta,\alpha} \Vert \nabla u \Vert_{L^1(\Omega)} , \notag \\
(iii) & \ \ |b_i- b_j| \UUU  > \EEE   \alpha \lambda_{\theta,\alpha} \Vert \nabla u \Vert_{L^1(\Omega)} \ \ \ \ \text{for} \ 1 \le i < j \le J.
\end{align}
\end{lemma}

\MMM We briefly comment on the statement of Lemma \ref{lemma: key}. \EEE Property \eqref{eq:II-new}(i) is an estimate of Poincar\'e-type on the components $P_j$. In contrast to Theorem \ref{th: main poinc}, the estimate has the additional property that the difference of the translations can be \BBB controlled \EEE from below in terms of the parameter $\alpha$, \UUU see \eqref{eq:II-new}(iii). \EEE The choice $\alpha \gg 1$ then implies that the values of $u$ on different components $(P_j)_j$ are `well separated', see \eqref{eq:II-new}(i),(iii). This will eventually allow us to exploit ($g2$) in the proof of Theorem \ref{thm: translations} \MMM and to show the energy estimate \eqref{eq: main estimates}(i). \EEE 

\MMM  The main idea to achieve \eqref{eq:II-new}(i),(iii) is as  follows: note that the components and translations given by Theorem \ref{th: main poinc} (or even just subsets of them) do possibly not satisfy \eqref{eq:II-new}(iii). The strategy is to sort the indices into different groups by means of Lemma \ref{lemma: balls}  below such that (a) the translations in each group are close to each other (in terms of a constant $\lambda_{\theta,\alpha}$), and (b) the translations in different groups differ very much (in terms of $\alpha \lambda_{\theta,\alpha}$). Then a new partition is defined by combining the components of each group and by defining new translations accordingly. We point out that the grouping of the indices and the explicit choice of $\lambda_{\theta,\alpha}$ depend on $u$, but $\lambda_{\theta,\alpha}$ always lies in the interval $[1,C_{\theta,\alpha}]$ independent of $u$. \EEE

Note that this refined Poincar\'e estimate comes at the expense of two \emph{rest sets} $R_1$ and $R_2$. For $R_2$ we have  \eqref{eq:II-new}(ii) which again  means that the values of $u$ on each component $P_j$ and $R_2$ are `well separated'. Finally, for $R_1$ we will exploit that \UUU the $\mathcal{H}^{d-1}$-measure of \EEE its boundary is small in terms of $\theta$, cf. \eqref{eq: R1R2-new}(ii). We remark that the necessity of rest sets is obvious if one considers functions with dense image in $\R^m$: \UUU in fact, the image of $u$ restricted to $\bigcup_{j=1}^J P_j$ is contained in  $\bigcup_{j=1}^J B_r(b_j)$ with $r = \lambda_{\theta,\alpha} \Vert \nabla u \Vert_{L^1(\Omega)}$ which does not cover $\R^m$ for $\alpha \ge 2$, cf.\ \eqref{eq:II-new}(i),(iii). \EEE

We defer the proof of Lemma \ref{lemma: key} to Section \ref{sec: piecewise Poin} and proceed with the proofs of Theorem \ref{thm: translations}--Corollary \ref{cor: truncations}.

\begin{proof}[Proof of Theorem \ref{thm: translations}]
We apply Lemma \ref{lemma: key} on $u$ for   $\alpha = \UUU 8\theta^{-1}c_3 \EEE +6$ \EEE to obtain a partition of $\Omega$, consisting of the sets $(P_j)_{j=1}^J$ and $R:= R_1 \cup R_2$, and to get translations $(b_j)_{j=1}^J$ such that \eqref{eq: R1R2-new}-\eqref{eq:II-new} hold. Then \eqref{eq: R1R2-new}  and the fact that $u \in GSBV^p_M(\Omega;\R^m)$ imply \eqref{eq: main estimates}(iv) and \eqref{eq: main estimates}(v). We define $t_j = b_j$ if $|b_j|>\lambda_{\theta,\alpha} \Vert \nabla u \Vert_{L^1(\Omega)}$ and $t_j = 0$ else. Note that  at most one $t_j$ is zero. Indeed, $t_{j_1} = t_{j_2} = 0$ for $j_1 \neq j_2$ would imply
$$ |b_{j_1} - b_{j_2}| \le 2\lambda_{\theta,\alpha} \Vert \nabla u \Vert_{L^1(\Omega)}.$$
 In view of  $\alpha \UUU \ge \EEE 2$,  however, this contradicts \eqref{eq:II-new}(iii).   Define $v:= \sum_{j=1}^J(u - t_j) \chi_{P_j}$. We show $\mathcal{L}^d(\lbrace u = 0 \rbrace \setminus  \lbrace v = 0 \rbrace) =0$. \UUU Since  $v = 0$ on $R$, it suffices to show $\mathcal{L}^d((\lbrace u = 0 \rbrace \setminus \lbrace v = 0 \rbrace)\cap P_j) = 0$   for each $j=1,\ldots,J$. Suppose that $\mathcal{L}^d(\lbrace u = 0 \rbrace \cap P_{j})>0$. Then $|b_{j}| \le  \lambda_{\theta,\alpha} \Vert \nabla u \Vert_{L^1(\Omega)}$ by \eqref{eq:II-new}(i), i.e., $t_{j} = 0$. This implies $v=u$ on $P_j$ and thus $\lbrace v = 0 \rbrace \cap P_j = \lbrace u = 0 \rbrace \cap P_j$. \EEE

  By \eqref{eq:II-new} and the fact that $|t_j - b_j| \le \lambda_{\theta,\alpha} \Vert \nabla u \Vert_{L^1(\Omega)}$ for $j=1,\ldots,J$ we obtain
\begin{align}\label{eq:II-new2}
(i) &  \ \ \Vert v \Vert_{L^\infty(\Omega)} \le 2\lambda_{\theta,\alpha} \Vert \nabla u \Vert_{L^1(\Omega)}\notag\\
(ii) & \ \ \min\nolimits_{1 \le j \le J} \  {\rm ess \, inf} \lbrace  |u(x) -  t_j|: \, x \in R_2 \rbrace  \ge  (\alpha-1) \lambda_{\theta,\alpha} \Vert \nabla u \Vert_{L^1(\Omega)} , \notag \\
(iii) & \ \ |t_i- t_j| \ge  (\alpha-2)\lambda_{\theta,\alpha} \Vert \nabla u \Vert_{L^1(\Omega)} \ \ \ \ \text{for} \ 1 \le i < j \le J.
\end{align}
Note that $\Vert \nabla u \Vert_{L^1(\Omega)} \le C\Vert \nabla u \Vert_{L^p(\Omega)} \le CM^{1/p}$ by H\"older's inequality for a constant $C$ depending on $\Omega$.  This along with \eqref{eq:II-new2}(i) and $\lambda_{\theta,\alpha} \le C_{\theta,\alpha}$ yields \eqref{eq: main estimates}(iii) \UUU for $C_{\theta,M}$ sufficiently large. \EEE The fact that $u \in GSBV^p(\Omega;\R^m)$, \eqref{eq: main estimates}(iii), and \eqref{eq: main estimates}(v) yield $v \in SBV^p(\Omega;\R^m) \cap L^\infty(\Omega;\R^m)$.

\UUU It remains to show \eqref{eq: main estimates}(i),(ii) and \eqref{eq: extra energy}. \EEE  Fix $E \in \mathcal{E}_\Omega$. For the bulk integral we obtain by ($f2$), \eqref{eq: main estimates}(iv), and the fact that $\nabla v = \nabla u$ on $\Omega \setminus R$
\begin{align}\label{eq: bulk}
\int_\Omega f(x,\nabla v) \, dx & =  \int_{\Omega \setminus R} f(x,\nabla v) \, dx + \int_{R} f(x,0)\, dx \le \int_{\Omega \setminus R} f(x,\nabla u) \, dx + c_2\mathcal{L}^d(R)\notag \\&\le \int_{\Omega} f(x,\nabla u) \, dx + C_M \theta. 
\end{align}
(As usual, the generic constant $C_M$ may vary from step to step.) For brevity we define  $\Gamma: =  \big(  \bigcup_{j=1}^{J} \partial^* P_j \cup \partial^* R\big) \cap \Omega$. We can split the surface integral into
\begin{align}\label{eq: surf1}
\int_{J_v} g(x,[v],\nu_v)\, d\mathcal{H}^{d-1} =  T_1 + T_2 := \int_{J_v \setminus \Gamma} g(x,[v],\nu_v)\, d\mathcal{H}^{d-1} +  \int_{J_v \cap  \Gamma} g(x,[v],\nu_v)\, d\mathcal{H}^{d-1}.
\end{align}
We start with $T_1$. Recall that the sets $(P_j)_{j=1}^{J}$ and $R$ form a Caccioppoli partition of $\Omega$. By the fact that $v = 0$ on $R$, $v = u - t_j$ on $(P_j)^1,$ and the structure theorem for Caccioppoli partitions (Theorem \ref{th: local structure}) we find
\begin{align}\label{eq: surf2}
T_1 & =  \int_{J_v \setminus \Gamma} g(x,[v],\nu_v)\, d\mathcal{H}^{d-1} = \sum\nolimits_{j=1}^{J} \int_{J_v \cap (P_j)^1}  g(x,[v],\nu_v)\, d\mathcal{H}^{d-1}\notag \\& = \sum\nolimits_{j=1}^{J} \int_{J_u \cap (P_j)^1}  g(x,[u],\nu_u)\, d\mathcal{H}^{d-1} \le  \int_{J_u \setminus \Gamma} g(x,[u],\nu_u)\, d\mathcal{H}^{d-1}.
\end{align} 
To estimate $T_2$, we split $\Gamma$ into the sets  (a) $\partial^* P_i \cap \partial^* P_j$, $1 \le i < j \le J$, (b) $\partial^* P_j \cap \partial^* R_2$, $1 \le j \le J,$ and (c) $\partial^* P_j \cap \partial^* R_1$, $1 \le j \le J$.

(a) First, \eqref{eq:II-new}(i),(iii) show that $J_u \supset \partial^* P_i \cap \partial^* P_j$ up to an $\mathcal{H}^{d-1}$-negligible set. \BBB We \EEE choose the orientation of $\nu_u(x)$ for $x \in \partial^* P_i \cap \partial^* P_j$ such that $u^+(x)$ coincides with the trace of $u\chi_{P_i}$ at $x$ and  $u^-(x)$ coincides with the trace of $u\chi_{P_j}$ at $x$. (The traces have to be understood in the sense of \cite[Theorem 3.77]{Ambrosio-Fusco-Pallara:2000}.) Moreover, we suppose that $\nu_v= \nu_u$ on $J_v \cap \partial^* P_i \cap \partial^* P_j$. Then we obtain by definition 
$$[v](x) = v^+(x) - v^-(x) =  (u^+(x) - t_i) - (u^-(x) - t_j) = [u](x) - (t_i - t_j) $$
for $\mathcal{H}^{d-1}$-a.e.\ $x \in J_v \cap \partial^* P_i \cap \partial^* P_j$. By \eqref{eq:II-new2}(i),(iii) we get
\begin{align*}
|[v](x)| &= | [u](x) -  (t_i - t_j)| \le 2\Vert v \Vert_{L^\infty(\Omega)}  \le  4\lambda_{\theta,\alpha} \Vert \nabla u \Vert_{L^1(\Omega)},\\
| [u](x)| &\ge  |t_i - t_j| - 2 \Vert v \Vert_{L^\infty(\Omega)} \ge   (\alpha-6) \lambda_{\theta,\alpha} \Vert \nabla u \Vert_{L^1(\Omega)}.
\end{align*} 
Using $\alpha = \UUU 8\theta^{-1}c_3 \EEE +6$ \UUU and again \eqref{eq:II-new2}(i), \EEE we derive for $\mathcal{H}^{d-1}$-a.e.\ $x \in \UUU J_v \cap \EEE \partial^* P_i \cap \partial^* P_j$ 
\begin{align}\label{eq: surf3}
|[v](x)|  \le |[v](x)| + \UUU 2 \theta^{-1} \EEE  \Vert v \Vert_{L^\infty(\Omega)}  \le \UUU 8 \theta^{-1} \EEE\lambda_{\theta,\alpha} \Vert \nabla u \Vert_{L^1(\Omega)} \le \frac{\UUU 8 \theta^{-1} \EEE}{\alpha-6}|[u](x)| \UUU = \EEE  \frac{1}{c_3} |[u](x)|.
\end{align}
(We include an additional addend $\UUU \frac{2}{\theta}  \EEE \Vert v \Vert_{L^\infty(\Omega)}$ since this will be convenient for the proof of \eqref{eq: extra energy}.)
 
(b) Similarly as before, \eqref{eq:II-new}(i),(ii) show that $J_u \supset \partial^* P_j \cap \partial^* R_2$ up to an $\mathcal{H}^{d-1}$-negligible set. We choose the orientation of $\nu_u(x)$ for $x \in \partial^* P_j \cap \partial^* R_2$ such that $u^+(x)$ coincides with the trace of $u\chi_{P_j}$ at $x$ and  $u^-(x)$ coincides with the trace of $u\chi_{R_2}$ at $x$. Moreover, we suppose that $\nu_v= \nu_u$ on $J_v \cap \partial^* P_j \cap \partial^* R_2$. \UUU Since $v=0$ on $R_2$, we then  obtain \EEE 
$$[v](x) = v^+(x) =  u^+(x) - t_j$$
for $\mathcal{H}^{d-1}$-a.e.\ $x \in  J_v \cap \partial^* P_j \cap \partial^* R_2$. By \eqref{eq:II-new2} we get
\begin{align*}
|[v](x)| &= | u^+(x) -  t_j| \le \Vert v \Vert_{L^\infty(\Omega)}  \le  2\lambda_{\theta,\alpha} \Vert \nabla u \Vert_{L^1(\Omega)},\\
| [u](x)| &\ge  |u^-(x) - t_j| -  |u^+(x) - t_j|\ge |u^-(x) - t_j| - \Vert v \Vert_{L^\infty(\Omega)}  \ge   (\alpha-3) \lambda_{\theta,\alpha} \Vert \nabla u \Vert_{L^1(\Omega)}.
\end{align*} 
Recalling $\alpha = \UUU 8\theta^{-1}c_3 \EEE +6$, we deduce for $\mathcal{H}^{d-1}$-a.e.\ $x \in J_v \cap \partial^* P_j \cap \partial^* R_2$ 
\begin{align}\label{eq: surf5} 
|[v](x)| \le  |[v](x)| + \UUU \theta^{-1} \EEE\Vert v \Vert_{L^\infty(\Omega)}   \le \UUU 4\theta^{-1} \EEE \lambda_{\theta,\alpha} \Vert \nabla u \Vert_{L^1(\Omega)} \le \frac{\UUU 4\theta^{-1} \EEE}{\alpha-3}|[u](x)| \le  \frac{1}{c_3} |[u](x)|.
\end{align}
 (As before, the additional addend $\UUU \theta^{-1} \EEE \Vert v \Vert_{L^\infty(\Omega)}$ will be needed for the proof of \eqref{eq: extra energy}.)

(c) Finally, for \RRR $\partial^* R_1$ we use \eqref{eq: R1R2-new}(ii) and $\mathcal{H}^{d-1}(J_u)\le M$ to find \EEE
\begin{align}\label{eq: surf6}
\mathcal{H}^{d-1}(\partial^* R_1) \le C_M \theta.
\end{align}

\UUU We are now in a position to show \eqref{eq: main estimates}(i). \EEE From \eqref{eq: surf3}-\eqref{eq: surf5} we get that $c_3 |[v](x)| \le  |[u](x)|$ for $\mathcal{H}^{d-1}$-a.e.\ $x \in (\Gamma \cap J_v) \setminus \partial^* R_1$. \BBB  Using this, $\nu_u = \nu_v$  $\mathcal{H}^{d-1}$-a.e.\  on $(\Gamma \cap J_v) \setminus \partial^* R_1$,  and \eqref{eq: surf6}, we derive  by ($g2$) and ($g3$) \EEE
\begin{align*}
T_2 & = \int_{ (J_v \cap  \Gamma) \setminus \partial^* R_1} g(x,[v],\nu_v)\, d\mathcal{H}^{d-1} +  \int_{ (J_v \cap  \Gamma) \cap \partial^* R_1} g(x,[v],\nu_v)\, d\mathcal{H}^{d-1} \\
& \le  \int_{(J_u \cap \Gamma) \setminus \partial^* R_1 } g(x,[u],\nu_u) \, d\mathcal{H}^{d-1}  +   \RRR c_5 C_M \theta \EEE \le  \int_{J_u \cap \Gamma} g(x,[u],\nu_u) \, d\mathcal{H}^{d-1} + C_M \theta.
\end{align*}
This along with   \eqref{eq: surf1}-\eqref{eq: surf2} yields
\begin{align}\label{eq: g for later}
\int_{J_v} g(x,[v],\nu_v)\, d\mathcal{H}^{d-1} \le \int_{J_u} g(x,[u],\nu_u) \, d\mathcal{H}^{d-1} + C_M \theta.
\end{align}
Now \eqref{eq: bulk} and  \eqref{eq: g for later} give \eqref{eq: main estimates}(i). Choosing specifically $g = c_4$,  \eqref{eq: g for later} also yields \eqref{eq: main estimates}(ii). Finally, the same calculation can be repeated for  $v' := \sum_{j=1}^J(u - t'_j) \chi_{P_j}$, where $(t'_j)_{j=1}^J$ satisfy $|t_j - t_j'| \le  \theta^{-1}  \Vert v\Vert_{L^\infty(\Omega)}$ for $j=1,\ldots,J$. Indeed, in this case we still have $c_3|[v'](x)| \le |[u](x)|$   for $\mathcal{H}^{d-1}$-a.e.\ $x \in (\Gamma \cap J_{v'}) \setminus \partial^* R_1$, see \eqref{eq: surf3} and \eqref{eq: surf5}.   \end{proof}

\begin{remark}\label{rem: easy}
{\normalfont
We recall from the proof that at most one translation $t_j$ is zero. Say, without restriction, $t_1 = 0$.   By \eqref{eq:II-new2}(i),(iii) we then find for all $j\ge 2$ and  almost all $x \in P_j$
 \begin{align*}
 |u(x)|= |u(x) - t_1| \ge |t_j - t_1| -  |u(x) - t_j| \ge (\alpha-4)\lambda_{\theta,\alpha} \Vert \nabla u\Vert_{L^1(\Omega)} \ge \UUU c_3\theta^{-1} \Vert \nabla u\Vert_{L^1(\Omega)}. \EEE 
 \end{align*}
where the last step follows from $\alpha \ge 4 + c_3/\theta$ and $\lambda_{\theta,\alpha} \ge 1$ (see Lemma \ref{lemma: key}).
}
\end{remark}

\begin{proof}[Proof of Corollary \ref{cor: translations-bdy}]
As  $u \in GSBV^p_M(\Omega';\R^m)$ and $\Vert \nabla h \Vert^p_{L^p(\Omega')} \le M$, we observe that $u - h \in GSBV^p_{2^pM}(\Omega';\R^m)$. We apply Theorem \ref{thm: translations} on $u-h$ and find $\bar{v} := \sum_{j=1}^J(u - h - t_j) \chi_{P_j}$ such that \UUU \eqref{eq: main estimates}(ii)-(v) hold  with $\bar{v}$ in place of $v$. \EEE We also note that \eqref{eq: g for later} is satisfied with $\bar{v}$ in place of $v$ since $J_{u-h} = J_u$ and $[u-h]=[u]$ on $J_u$. As $u-h = 0$ on $\Omega' \setminus \overline{\Omega}$ and $\lbrace \bar{v}  =0 \rbrace \supset \lbrace u - h =0 \rbrace$, we get $\bar{v} = 0$ on  $\Omega' \setminus \overline{\Omega}$. This implies that $t_j = 0$ for each $P_j$ intersecting $\Omega' \setminus \overline{\Omega}$. As at most one $t_j$ is zero, see Remark \ref{rem: easy}, at most one component $P_j$ intersects $\Omega' \setminus \overline{\Omega}$.

  We define $v = \bar{v} + h = h \chi_R + \sum_{j=1}^J(u - t_j) \chi_{P_j} \in SBV^p(\Omega';\R^m)$. Clearly, $v = h$ on $\Omega' \setminus \overline{\Omega}$ as $\bar{v} = 0$ on  $\Omega' \setminus \overline{\Omega}$. Properties \eqref{eq: small set main-boundary}(ii)-(v) follow directly from \eqref{eq: main estimates}(ii)-(v) \BBB (with $\bar{v}$ in place of $v$). \EEE To see \eqref{eq: small set main-boundary}(i), we compute by  ($f2$), \eqref{eq: small set main-boundary}(iv), and  \eqref{eq: g for later} (with $\bar{v}$ in place of $v$)
\begin{align*} 
E(v) &= \int_{\Omega'} f(x,\nabla v) \, dx + \int_{J_v} g(x,[v],\nu_v)\, d\mathcal{H}^{d-1} \\
&\le   \int_{\Omega' \setminus R} f(x,\nabla u) \, dx + \int_{R} f(x,\nabla h)\, dx + \int_{J_{\bar{v}}} g(x,[\bar{v}],\nu_{\bar{v}})\, d\mathcal{H}^{d-1}\\
&\le \int_{\Omega'} f(x,\nabla u) \, dx + c_2(\mathcal{L}^d(R) + \Vert \nabla h \Vert_{L^p(R)}^p)  + \int_{J_u} g(x,[u],\nu_u) \, d\mathcal{H}^{d-1} + C_M \theta  \\
& \le E(u)  + C_M \theta   + C_M \Vert \nabla h \Vert_{L^p(R)}^p
\end{align*}
for all $E \in \mathcal{E}_{\Omega'}$. Similarly, also \eqref{eq: extra energy2} follows as \eqref{eq: g for later} is still applicable in this case.   
\end{proof}

 \begin{proof}[Proof of Corollary \ref{cor: truncations}]
 \UUU Let $M$ large enough such that $\Vert \nabla h \Vert^p_{L^p(\Omega')} \le M$, $u \in GSBV^p_M(\Omega';\R^m)$. \EEE We apply Corollary \ref{cor: translations-bdy}  for $u$ and $\theta_k = 1/k$ for each $k \in \N$ to obtain functions $u_k:= h\chi_{R^k}  + \sum_{j=1}^{J^k}(u-t^k_j) \chi_{P^k_j}$.  \UUU They satisfy $u_k = h$ on $\Omega' \setminus \overline{\Omega}$ and $u_k \in SBV^p(\Omega';\R^m) \cap L^p(\Omega';\R^m)$ by \eqref{eq: small set main-boundary}(iii). Moreover, \EEE we have  $\limsup_{k \to \infty} E(u_k) \le E(u)$ by \eqref{eq: small set main-boundary}(i) \RRR and \EEE \eqref{eq: small set main-boundary}(iv).
 
 We need to check that $u_k \to u$ in measure on $\Omega'$. For $k$ sufficiently large such that $\mathcal{L}^d(R^k)< \mathcal{L}^d(\Omega' \setminus \overline{\Omega})$, exactly one component intersects $\Omega' \setminus \overline{\Omega}$, say without restriction $P^k_1$. As $u_k = h$ on $\Omega' \setminus \overline{\Omega}$, this implies $t^k_1 = 0$ and thus $u_k = u$  on $P^k_1$. By Remark \ref{rem: easy} (applied on $u-h$) and $\theta_k = 1/k$ we then get $ |u(x) - h(x)| \ge \BBB c_3 \EEE k \Vert \nabla u - \nabla h\Vert_{L^1(\Omega')}$ \UUU for a.e.\ $x \in \Omega' \setminus (R^k \cup P_1^k)$. As $u-h$ is finite almost everywhere, we find $\mathcal{L}^d(\Omega' \setminus (R^k \cup P_1^k)) \to 0$. (Note that, possibly slightly modifying $h$ inside $\Omega$, it is not restrictive to suppose $\Vert \nabla u - \nabla h\Vert_{L^1(\Omega')}>0$.) \EEE This along with $\mathcal{L}^d(R^k) \to 0$ by  \eqref{eq: small set main-boundary}(iv) yields $\mathcal{L}^d(\Omega' \setminus P^k_1) \to 0$. As $u_k = u$ on $ P^k_1$, we conclude $u_k \to u$ in measure on $\Omega'$. 
 \end{proof}

%

\subsection{Piecewise Poincar\'e inequality}\label{sec: piecewise Poin}

This section is devoted to the proof of Lemma \ref{lemma: key}. The reader may wish to skip this section on first reading and to proceed directly with the proof of Theorem \ref{th: comp} in Section \ref{sec: main proof}. As a preparation, we state the following elementary property.

\begin{lemma}[Covering with balls]\label{lemma: balls} 
Let $N \in \N$, $\gamma \ge 2$, and $R_0>0$. Then  each set of points $\lbrace x_1, \ldots, x_n \rbrace \subset \R^m$, $n \le N$, can be covered by finitely many pairwise disjoint balls $\lbrace B_{r_k}(y_k) \rbrace_{k=1}^M$, $M \le N$, $(y_k)_{k=1}^M \subset \R^m$, satisfying
\begin{align}\label{eq: balls}
r_k \in [R_0,(2\gamma)^N R_0] \ \ \    \text{for} \  k=1,\ldots,M, \ \ \  |y_i-y_j| \UUU > \EEE \gamma  \max_{k=1,\ldots,M} r_k \ \ \  \text{for} \  1 \le i < j \le M.
\end{align}
\end{lemma}

\begin{proof}
We prove the lemma by induction. Suppose that in step $l \in \N_0$ there exist  finitely many balls $\lbrace B_{r^l_k}(y^l_k) \rbrace_{k=1}^{M_l}$, $M_l \le N -l$, which cover $\lbrace x_1, \ldots, x_n \rbrace$ and satisfy $r^l_k \in [R_0,(2\gamma)^{l}R_0]$. For step $l=0$, we can take the balls centered at $\lbrace x_1, \ldots, x_n \rbrace$ with radius $R_0$.

If in some iteration step $l\le N-1$ we have  
\begin{align}\label{eq: balls2}
|y^l_i-y^l_j| \UUU > \EEE \gamma  \max_{k=1,\ldots,M_l} r^l_k \ \ \  \text{for} \ \ 1 \le i < j \le M_l,
\end{align}
 we have found a collection of balls covering $\lbrace x_1, \ldots, x_n \rbrace$ and satisfying \eqref{eq: balls}. We also observe that \eqref{eq: balls} and $\gamma \ge 2$ induce that the balls are pairwise disjoint. Otherwise, it is not restrictive to suppose that $|y^l_1-y^l_2| \UUU \le \EEE \gamma  \max_{k=1,\ldots,M_l} r^l_k$. Letting $r = 2 \gamma \, \max_{k=1,\ldots,M_l} r^l_k \le (2\gamma)^{l+1}R_0$, we observe that $B_r(y^l_1) \supset B_{r_1^l}(y_1^l) \cup B_{r_2^l}(y_2^l)$. We let $B_r(y^l_1)$ and $B_{r_k^l}(y_k^l)$, $3 \le k \le M_l$, be the collection of balls in iteration step $l+1$, \MMM whose number is $M_l - 1$ and thus at most $N-(l+1)$. \EEE

Now we observe that after at most $N-1$ iteration steps we have found a collection of balls such that \eqref{eq: balls2} holds. Indeed, in step $N-1$, the collection consists only of one ball. 
\end{proof}

We now proceed with the proof of Lemma \ref{lemma: key}.

\begin{proof}[Proof of Lemma \ref{lemma: key}]  Let $u \in GSBV^p(\Omega;\R^m)$,  $\alpha \ge 1$, and $0 < \theta < 1$ be given. Let $C_0 \ge 1$ be the constant from Theorem \ref{th: main poinc}.  Define \BBB $\beta = 6 \alpha (4\alpha)^{\theta^{-d}} $ for brevity. \EEE We first use Theorem \ref{th: main poinc} to define an auxiliary partition and corresponding translations such that \BBB estimates of type \EEE \eqref{eq:II-new}(i),(ii) are already satisfied (Step 1). Subsequently, \BBB we apply Lemma \ref{lemma: balls}  to \EEE pass to a coarser partition and we define the translations suitably to ensure also  \eqref{eq:II-new}(iii) (Step 2).

\noindent \textit{Step 1 (Auxiliary partition).}  The goal of this step is to find two disjoint \emph{rest sets} $R_1,R_2 \subset \Omega$ satisfying  
\begin{align}\label{eq: R1R2} 
(i) & \ \ \mathcal{L}^d(R_1 \cup R_2) \le C_\Omega {\theta} \, \mathcal{H}^{d-1}(J_u \cup \partial \Omega), \notag \\
(ii) & \ \ \mathcal{H}^{d-1}(\partial^* R_1) \le C_\Omega\theta \, \mathcal{H}^{d-1}(J_u \cup \partial \Omega), \ \ \ \ \ \mathcal{H}^{d-1}(\partial^* R_2) \le C_\Omega \mathcal{H}^{d-1}(J_u \cup \partial \Omega),
\end{align}
a finite Caccioppoli partition  $\Omega = \bigcup^{J_{\rm a}}_{i=1} P^{\rm a}_i \cup R_1 \cup R_2$  for an index $J_{\rm a} \in \N$ with $J_{\rm a} \le \theta^{-d}$,  and corresponding translations $(b^{\rm a}_i)_{i=1}^{J_{\rm a}}$ such that \BBB we have with \EEE $v_{\rm a}  := \sum_{i =1}^{J_{\rm a}} (u - b^{\rm a}_i) \chi_{P^{\rm a}_i} $ 
\begin{align}\label{eq:I}
(i) & \ \ \sum\nolimits_{i= 1}^{J_{\rm a} } \mathcal{H}^{d-1}(\partial^* P^{\rm a}_i) \le C_\Omega\mathcal{H}^{d-1}(J_u \cup \partial \Omega), \notag\\
(ii) &  \ \ \Vert v_{\rm a}  \Vert_{L^\infty(\Omega)} \le 4C_{0}\beta^{K_\theta}\Vert  \nabla u \Vert_{L^1(\Omega)}, \notag\\
(iii) & \ \ \min\nolimits_{1 \le i \le J_{\rm a}} \  {\rm ess \, inf} \lbrace  |u(x) -  b_i^{\rm a}|: \ x \in R_2 \rbrace  \ge 2C_0\beta^{K_\theta+1}\Vert  \nabla u \Vert_{L^{1}(\Omega)}
\end{align}
for some $K_\theta \in \N$, $K_\theta \le \theta^{-1}$. Here, $C_\Omega>0$ is a constant only depending on $\Omega$. 

\emph{Proof of Step 1.} We apply Theorem \ref{th: main poinc} on $u$ to find an ordered Caccioppoli partition $(P'_j)_j$ of $\Omega$ and corresponding translations $(b'_j)_j \subset \R^m$ such that
\begin{align}\label{eq: kornpoinsharp4}
(i) & \ \   \sum\nolimits_{j=1}^\infty \mathcal{H}^{d-1}( \partial^* P'_j)  \le  2 \mathcal{H}^{d-1}(J_u\cup \partial \Omega) +1 \le C_\Omega\mathcal{H}^{d-1}(J_u\cup \partial \Omega), \notag \\
(ii) &\ \ \Vert u - b'_j \Vert_{L^{\infty}(P_j')} \le C_0 \Vert  \nabla u \Vert_{L^1(\Omega)} \ \ \ \text{for all } j \in \N,
\end{align}
where $C_\Omega$ depends only on $\Omega$. Let $J_{\rm a} \in \N $ be the largest index such that $\mathcal{L}^d(P'_{J_{\rm a} }) \ge \theta^d \mathcal{L}^d(\Omega)$. Then $J_{\rm a}  \le \theta^{-d}$.  (Recall that the partition is assumed to be ordered.) By the isoperimetric inequality \UUU and \eqref{eq: kornpoinsharp4}(i) \EEE we have
\begin{align}\label{eq: volume}
\sum\nolimits_{j > J_{\rm a} } \mathcal{L}^d(P'_j)& \le  (\theta^d \mathcal{L}^d(\Omega))^{1/d} \sum\nolimits_{j > J_{\rm a} } \big(\mathcal{L}^d(P'_j)\big)^{1 - 1/d} \le C_\Omega \, \theta \,  \sum\nolimits_{j \ge 1} \mathcal{H}^{d-1}(\partial^* P'_j) \notag\\ 
& \le C_\Omega \theta \,  \mathcal{H}^{d-1}(J_u\cup \partial \Omega).
\end{align}
 We now introduce a decomposition of the components $(P_j')_{j > J_{\rm a} }$ according to the difference of the translations: for  $k \in \N$ we define the sets of indices
\begin{align}\label{eq: kornpoinsharp5}
\begin{split}
&\mathcal{J}^0 = \big\{  j > J_{\rm a} :  \min\nolimits_{1 \le i \le J_{\rm a} } |b'_j -b'_i| \le 3C_0\beta \Vert  \nabla u \Vert_{L^1(\Omega)}  \big\}, \\
& \mathcal{J}^k = \big\{ j > J_{\rm a} :  3C_0\beta^{k} \Vert  \nabla u \Vert_{L^1(\Omega)} < \min\nolimits_{1 \le i \le J_{\rm a} }|b'_j -b'_i| \le 3C_0\beta^{k+1}\Vert  \nabla u \Vert_{L^1(\Omega)}  \big\}.
\end{split}
\end{align}
\BBB Let \EEE $s_k = \sum_{j \in \mathcal{J}^k} \mathcal{H}^{d-1}(\partial^* P'_j)$ for $k \in \N_0$. In view of  \eqref{eq: kornpoinsharp4}(i), we find some $K_\theta \in \N$, $K_\theta \le \theta^{-1}$, such that $s_{K_\theta} \le C_\Omega\theta \mathcal{H}^{d-1}(J_u\cup \partial \Omega)$. Define
\begin{align}\label{eq: R1/R2-def}
R_1 := \bigcup\nolimits_{j \in \mathcal{J}^{K_\theta}} P'_j, \ \ \ \ \ \ \ R_2 = \bigcup\nolimits_{k> K_\theta}\bigcup\nolimits_{j \in \mathcal{J}^{k}} P'_j.
\end{align} 
The choice of $K_\theta$, \eqref{eq: kornpoinsharp4}(i), and \eqref{eq: volume} show \eqref{eq: R1R2}. We introduce a Caccioppoli partition $(P^{\rm a}_i)_{i=1}^{J_{\rm a} }$ of $\Omega \setminus (R_1 \cup R_2)$ by combining different components of $(P'_j)_{j \ge 1}$: we decompose the indices in $\bigcup_{k=0}^{K_\theta-1} \mathcal{J}^k$ into sets $\mathcal{I}_i$ with $\bigcup_{i=1}^{J_{\rm a} } \mathcal{I}_i = \bigcup_{k=0}^{K_\theta-1} \mathcal{J}^k$ according to the following rule: an index $j\in \mathcal{J}^k$ is assigned to $\mathcal{I}_i$ when $i$ is the smallest index \UUU in $\lbrace 1, \ldots, J_{\rm a} \rbrace$ \EEE such that the minimum in \eqref{eq: kornpoinsharp5} is attained. Let $P^{\rm a}_i = P'_i \cup \bigcup_{j \in \mathcal{I}_i} P'_j$ for $1 \le i \le J_{\rm a} $ \BBB and observe that the sets form a partition of $\Omega \setminus (R_1 \cup R_2)$. \EEE 

We now show  \eqref{eq:I}. First, \eqref{eq:I}(i) holds by \eqref{eq: kornpoinsharp4}(i). We  define $b^{\rm a}_i = b'_i$ for $1 \le i \le J_{\rm a} $. Let  $v' := u -  \sum_{j \ge 1}   b_j'\chi_{P_j'}$ and $v_{\rm a}  := \sum_{i=1}^{J_{\rm a} } (u -  b_i^{\rm a}) \chi_{P_i^{\rm a}}$. We find by the definition of $\mathcal{I}_i \UUU \subset \bigcup_{k=0}^{K_\theta-1} \mathcal{J}^k\EEE$ and  \eqref{eq: kornpoinsharp5} 
$$\Vert v_{\rm a}  - v' \Vert_{L^\infty(P^{\rm a}_i)} \le 3C_0\beta^{K_\theta}\Vert  \nabla u \Vert_{L^1(\Omega)} $$
for $i=1,\ldots, J_{\rm a} $. By \eqref{eq: kornpoinsharp4}(ii) we compute for each $i=1,\ldots, J_{\rm a} $
\begin{align*} 
\Vert v_{\rm a}  \Vert_{L^\infty(P_i^{\rm a})} &\le \Vert  v_{\rm a}  - v' \Vert_{L^\infty(P_i^{\rm a})} + \Vert v' \Vert_{L^\infty(P_i^{\rm a})}   \le 3C_0\beta^{K_\theta}\Vert  \nabla u \Vert_{L^1(\Omega)}  + C_0 \Vert  \nabla u \Vert_{L^1(\Omega)} \\
&  \le 4  C_0\beta^{K_\theta} \Vert  \nabla u \Vert_{L^1(\Omega)}.
\end{align*} 
This  yields \eqref{eq:I}(ii). Finally, we show \eqref{eq:I}(iii).  Fix $1 \le i \le J_{\rm a} $. For $\mathcal{L}^d$-a.e.\ $x \in R_2$, we choose $j > J_{\rm a}$ such that $x \in P_j' \subset R_2 $  (recall \eqref{eq: R1/R2-def}). Then we  compute  by \eqref{eq: kornpoinsharp4}(ii),   \eqref{eq: kornpoinsharp5}, and the fact $j \in \bigcup_{k > K_\theta}\mathcal{J}^k$
\begin{align*}
|u(x) -  b_i^{\rm a}| &\ge |b_i^{\rm a} - b_j'| -  |u(x) - b'_j|=  |b_i' - b_j'| -  |u(x) - b'_j|  \\& \ge  3C_0\beta^{K_\theta+1}\Vert  \nabla u \Vert_{L^1(\Omega)}  - C_0\Vert  \nabla u \Vert_{L^1(\Omega)} \ge  2C_0\beta^{K_\theta+1}\Vert  \nabla u \Vert_{L^1(\Omega)}. 
\end{align*}
This concludes the proof of Step 1.

\noindent \textit{Step 2 (Passage to coarser partition).} We now pass to a coarser Caccioppoli partition: there exists a partition $\Omega = \bigcup^{J}_{j=1} P_j  \cup  R_1 \cup R_2$ with $J \le J_{\rm a}  \le \theta^{-d}$ and $\bigcup_{j=1}^{J} \partial^* P_j \subset \bigcup_{i=1}^{J_{\rm a} } \partial^* P^{\rm a}_i$ up to an $\mathcal{H}^{d-1}$-negligible set, as well as corresponding translations $(b_j)_{j=1}^{J}$ such that  for some $\lambda_{\theta,\alpha}>0$
\begin{align}\label{eq:II}
(i) & \ \ \sum\nolimits_{j= 1}^J \mathcal{H}^{d-1}(\partial^* P_j) \le C_\Omega\mathcal{H}^{d-1}(J_u\cup \partial \Omega), \notag\\
(ii) &  \ \ \Vert u - b_j \Vert_{L^\infty(P_j)} \le \lambda_{\theta,\alpha} \Vert \nabla u \Vert_{L^1(\Omega)} \ \ \ \ \text{for} \ 1 \le j \le J, \notag\\
(iii) & \ \ \min\nolimits_{1 \le j \le J} \  {\rm ess \, inf} \lbrace  |u(x) -  b_j|: \ x \in R_2 \rbrace  \ge  \alpha \lambda_{\theta,\alpha}\Vert  \nabla u \Vert_{L^1(\Omega)} , \notag \\
(iv) & \ \ |b_i- b_j| \UUU > \EEE \alpha \lambda_{\theta,\alpha}\Vert  \nabla u \Vert_{L^1(\Omega)} \ \ \ \ \text{for} \ 1 \le i < j \le J.
\end{align}

\emph{Proof of Step 2.} We apply Lemma \ref{lemma: balls} on the points $(b_i^{\rm a})_{i=1}^{J_{\rm a} }$ for $R_0 = 4C_{0}\beta^{K_\theta} \Vert \nabla u \Vert_{L^1(\Omega)}$ and  $\gamma = 2\alpha$. We obtain finitely many pairwise disjoint balls $\lbrace B_{r_j}(y_j) \rbrace_{j=1}^{J}$, $J \le J_{\rm a}  \le \theta^{-d}$, which cover   $(b_i^{\rm a})_{i=1}^{J_{\rm a} }$ and satisfy
\begin{align}\label{eq: rk}
r_j \in [4C_{0}\beta^{K_\theta} \Vert \nabla u \Vert_{L^1(\Omega)}, 4C_{0}\beta^{K_\theta} \Vert \nabla u \Vert_{L^1(\Omega)} (4\alpha)^{J_{\rm a} }] \text{ for $j=1,\ldots,J$}
\end{align}
as well as
\begin{align}\label{eq: mainpointstep2}
|y_i-y_k| \UUU > \EEE 2\alpha  \max_{j=1,\ldots,J} r_j \ \ \  \text{for} \ \ 1 \le i < k \le J.
\end{align}
We set
\begin{align}\label{eq: lambdadef}
\lambda_{\theta, \alpha} = 2\Vert  \nabla u \Vert_{L^1(\Omega)}^{-1}\max_{j=1,\ldots,J} r_j
\end{align}
and note that  $   8C_{0}\beta^{K_\theta} \le \lambda_{\theta, \alpha}  \le 8\delta C_{0}\beta^{K_\theta}$ by \eqref{eq: rk}, where for brevity we set $\delta := (4\alpha)^{\theta^{-d}}$. As the balls are pairwise disjoint, each $b_i^{\rm a}$ is contained in exactly one ball. We define $b_j = y_j$ for $j=1,\ldots,J$ and introduce the sets
\begin{align}\label{eq: Lk}
\mathcal{L}_j = \lbrace i: \ b_i^{\rm a} \in B_{r_j}(b_j) \rbrace, \ \ \ \ \ \ P_j = \bigcup\nolimits_{i \in \mathcal{L}_j} P^{\rm a}_i.
\end{align}
Then the components $(P_j)_{j=1}^{J}$ form a Caccioppoli partition of $\Omega \setminus (R_1 \cup R_2)$ which is coarser than $(P^{\rm a}_i)_i$. Note that \eqref{eq:II}(i) holds by \eqref{eq:I}(i).  

We now show \eqref{eq:II}(ii)-(iv). First, \eqref{eq: mainpointstep2} and the definition of $\lambda_{\theta,\alpha}$  show \eqref{eq:II}(iv). Fix $P_j$ and $P^{\rm a}_i$ with $P^{\rm a}_i \subset P_j$. Then by \eqref{eq:I}(ii), \eqref{eq: lambdadef}-\eqref{eq: Lk},  and the fact that $4C_{0}\beta^{K_\theta} \le \frac{1}{2}\lambda_{\theta, \alpha}$
\begin{align*}
\Vert u - b_j \Vert_{L^\infty(P_j \cap P^{\rm a}_i)} \le  \Vert v_{\rm a}  \Vert_{L^\infty(P_j \cap P^{\rm a}_i)} + |b^{\rm a}_i - b_j| \le 
4C_{0}\beta^{K_\theta}\Vert  \nabla u \Vert_{L^1(\Omega)} + r_j   \le \lambda_{\theta,\alpha} \Vert \nabla u \Vert_{L^1(\Omega)}.
\end{align*}
As $P^{\rm a}_i \subset P_j$ was arbitrary,  we get \eqref{eq:II}(ii). \BBB Recalling the definition of $\beta$ and $\delta$ we have  $\beta = 6 \alpha (4\alpha)^{\theta^{-d}}  = 6 \alpha \delta$. \EEE This along with \eqref{eq:I}(iii), \eqref{eq: lambdadef}-\eqref{eq: Lk}, and  $\lambda_{\theta, \alpha}  \le 8\delta C_{0}\beta^{K_\theta}$ yields 
\begin{align*}
\min_{1 \le j \le J}   {\rm ess \, inf} \lbrace  |u(x) -  b_j|:  x \in R_2 \rbrace  &\ge  \min_{1 \le i \le J_{\rm a} }  {\rm ess \, inf} \lbrace  |u(x) -  b_i^{\rm a}|: \, x \in R_2 \rbrace - \max_{1 \le j \le J} \max_{i \in \mathcal{L}_j} |b_j - b_i^{\rm a}| \\&
 \ge  2C_0\beta^{K_\theta+1}\Vert  \nabla u \Vert_{L^1(\Omega)}  - \max_{1 \le j \le J} r_j \\& \ge \big( 2C_0\beta^{K_\theta+1}  -  \lambda_{\theta,\alpha}/2 \big) \Vert  \nabla u \Vert_{L^1(\Omega)} \\
 &  \ge  \delta C_{0}\beta^{K_\theta}(12\alpha - 4) \Vert  \nabla u \Vert_{L^1(\Omega)}\ge 8\alpha\delta C_0\beta^{K_\theta}\Vert  \nabla u \Vert_{L^1(\Omega)}\\
 & \ge \alpha \lambda_{\theta,\alpha} \Vert  \nabla u \Vert_{L^1(\Omega)}. 
\end{align*}
 This shows \eqref{eq:II}(iii) and concludes   Step 2. \BBB Recall that we have \EEE $\lambda_{\theta,\alpha}  \ge 8C_0\beta^{K_\theta}\ge 1$ and $\lambda_{\theta,\alpha} \le 8\delta C_{0}\beta^{K_\theta}$. Thus, $\lambda_{\theta,\alpha} \le C_{\theta,\alpha}: = 8(4\alpha)^{\theta^{-d}}C_0 (6 \alpha (4\alpha)^{\theta^{-d}} )^{1/\theta}$. 
 
 The statement of the lemma now follows from \eqref{eq: R1R2}  and \eqref{eq:II}.
\end{proof}

 \subsection{Proof of Theorem \ref{th: comp}}\label{sec: main proof}

The proof of Theorem \ref{th: comp}  essentially relies on the following result.

\begin{theorem}[Existence of function $\psi$]\label{th: comp-aux}
Let $\Omega \subset \Omega' \subset \R^d$ be bounded Lipschitz domains. Let $(E_k)_k \subset \mathcal{E}_{\Omega'}$ and let $(h_k)_k \subset W^{1,p}(\Omega';\R^m)$ converging in  $L^p(\Omega';\R^m)$ to some $h \in W^{1,p}(\Omega';\R^m)$ such that $(|\nabla h_k|^p)_k$ are equi-integrable. Consider $(u_k)_k  \subset GSBV^p(\Omega';\R^m)$ with  $u_k = h_k$ on $\Omega' \setminus \overline{\Omega}$ and  $\sup_{k \in \N} E_k(u_k)  <+\infty$.

\noindent Then we find a  subsequence (not relabeled), modifications $(y_k)_k \subset GSBV^p(\Omega';\R^m)$ with $y_k = h_k$ on $\Omega' \setminus \overline{\Omega}$, and a continuous function $\psi:\R_+ \to \R_+$ with $\lim_{t \to \infty} \psi(t) = + \infty$ such that 
\begin{align}\label{eq: compi1-aux}
(i) & \ \   E_k(y_k) \le E_k(u_k) + \tfrac{1}{k}, \ \ \  \ \ \ \ \ \  \mathcal{H}^{d-1}(J_{y_k}) \le \mathcal{H}^{d-1}(J_{u_k}) + \tfrac{1}{k}, \notag \\
(ii) &   \ \  \sup\nolimits_{k\in \N}\int_{\Omega'} \psi(|y_k|)\, dx < + \infty, \notag\\
(iii) & \ \   \mathcal{L}^d(\lbrace\nabla y_k \neq \nabla u_k \rbrace) \le \tfrac{1}{k}.
\end{align}
\end{theorem}

Indeed, once Theorem \ref{th: comp-aux} is proved, Theorem \ref{th: comp} follows directly from Theorem \ref{th: GSBD comp} and ($f2$), ($g3$), apart from the \UUU property \EEE that $\nabla u_k  \rightharpoonup \nabla u$   weakly in $L^p(\Omega'; \R^{m\times d})$. To see the latter, we note  by ($f2$) that   we find $Z \in L^p(\Omega'; \R^{m\times d})$ such that $\nabla u_k  \rightharpoonup  Z$ weakly in $L^p(\Omega'; \R^{m\times d})$ (possibly up to a further subsequence). It   suffices to check   $Z = \nabla u$. To this end, we show that $\nabla u_k  \rightharpoonup  \nabla u$ weakly in $L^1(\Omega'; \R^{m\times d})$. Indeed, we have $\nabla y_k  \rightharpoonup  \nabla u$ weakly in $L^1(\Omega'; \R^{m\times d})$ and for each $\varphi \in L^\infty(\Omega';\R^{m\times d})$ we compute by \eqref{eq: compi1-aux}(iii) and H\"older's inequality
\begin{align*}
\Big|\int_{\Omega'} (\nabla u_k - \nabla y_k):\varphi \, dx  \Big|& \le  \Vert \varphi \Vert_\infty \int_{\lbrace\nabla y_k \neq \nabla u_k \rbrace} |\nabla u_k - \nabla y_k| \, dx \\&\le \big(\mathcal{L}^d(\lbrace\nabla y_k \neq \nabla u_k \rbrace)\big)^{1-1/p} \Vert \nabla u_k - \nabla y_k   \Vert_{L^p(\Omega')} \to 0.
\end{align*}

  We now proceed with the proof of Theorem \ref{th: comp-aux}.  \MMM We point out that the result does not simply follow from Corollary \ref{cor: translations-bdy}: to construct modifications $(y_k)_k$ satisfying  \eqref{eq: compi1-aux}(i), Corollary \ref{cor: translations-bdy} has to be applied along a sequence $\theta \to 0$ to obtain  piecewise translated configurations $(v^\theta_k)_{k,\theta}$. As $\theta \to 0$, unfortunately the uniform bound \eqref{eq: small set main-boundary}(iii) blows up, and the definition of the function $\psi$ is not immediate. As a remedy, we first pass to a limit $v^\theta$ for each fixed $\theta$ as $k \to \infty$, and then we show that $(v^\theta)_\theta$ are close to each other in a certain sense on the bulk part of the domain. This allows us to apply Lemma \ref{rig-lemma: concave function2} and to obtain the function $\psi$. Then,  $(y_k)_k$ can be chosen as a suitable diagonal sequence in $(v^\theta_k)_{k,\theta}$.  In this strategy, we  follow closely \cite[Theorem 6.1]{Solombrino} and \cite[Theorem 2.2]{Friedrich:15-2}.  Note, however, that some delicate adaptions are necessary  due to the fact that the energies may depend on the crack opening. \EEE

\begin{proof}[Proof of Theorem \ref{th: comp-aux}]
\UUU Consider a sequence $(E_k)_k \subset \mathcal{E}_{\Omega'}$. Let  $(h_k)_k \subset W^{1,p}(\Omega';\R^m)$ converging in  $L^p(\Omega';\R^m)$ to some $h \in W^{1,p}(\Omega';\R^m)$  such that $(|\nabla h_k|^p)_k$ are equi-integrable. Let $(u_k)_k  \subset GSBV^p(\Omega';\R^m)$ with  $u_k = h_k$ on $\Omega' \setminus \overline{\Omega}$   and $\sup_{k \in \N} E_k(u_k) \le \BBB C_* \EEE <+\infty$. Setting
\begin{align}\label{eq: M def}
M:= \frac{C_*}{c_1} + \frac{C_*}{c_4} + \sup_{k \in \N} \Vert \nabla h_k \Vert^p_{L^p(\Omega')},
\end{align}
we find $\Vert \nabla h_k \Vert^p_{L^p(\Omega')} \le M$ and $u_k \in GSBV^p_M(\Omega';\R^m)$ by ($f2$) and ($g3$). Define the decreasing sequence $\theta_l =  2^{-l}$  for $l\in \N$.
As we will pass to subsequences (not relabeled) several times in the proof, we emphasize that we will eventually only have  the inequality \EEE
\begin{align}\label{eq: theta0}
\theta_l \le 2^{-l}.
\end{align}

\emph{Step 1 (Application of  Corollary \ref{cor: translations-bdy}).} We apply  Corollary \ref{cor: translations-bdy} for $\theta_l$ and  $M$ on the functions $u_k$ and the boundary data $h_k$.  We find (finite) Caccioppoli partitions  $\Omega' = \bigcup_{j \ge 1} P_j^{k,l} \cup R^l_k$, and piecewise translated functions $(v^l_k)_k \subset SBV^p(\Omega';\R^m)$ defined by
\begin{align}\label{eq: comp4}
v^l_k :=  h_k +  \sum\nolimits_{j \ge 1} (u_k - t^{k,l}_j - h_k) \chi_{P_j^{k,l}} = h_k\chi_{R^l_k} +  \sum\nolimits_{j \ge 1} (u_k - t^{k,l}_j) \chi_{P_j^{k,l}},
\end{align}
 where $(t_j^{k,l})_{j \ge 1} \subset \R^m$  are  suitable translations. For notational convenience, we will also use the notation $P_0^{k,l} = R^l_k$ such that $(P_j^{k,l})_{j \ge 0}$ is a partition of $\Omega'$. From Corollary \ref{cor: translations-bdy}  we have $v^l_k = h_k$ on $\Omega' \setminus \overline{\Omega}$ for all $l, k \in \N$ and from  \eqref{eq: small set main-boundary}, \eqref{eq: M def}, \eqref{eq: comp4} we get
 \begin{align}\label{eq: comp1}
(i) & \ \ \Vert v^l_k - h_k \Vert_{L^\infty(\Omega')}  \le C_{\theta_l, M}, \notag \\ 
(ii) & \  \  \Vert \nabla v^l_k \Vert^p_{L^p(\Omega')}  \le \Vert \nabla u_k \Vert^p_{L^p(\Omega')}  + \Vert \nabla h_k \Vert^p_{L^p(R^l_k)}  \le  \BBB 2M, \EEE \notag \\
(iii) & \ \ \mathcal{L}^d(R^l_k) \le C_M\theta_l, \notag \\ 
(iv) & \ \ \mathcal{H}^{d-1}(J_{v^l_k}) \le   \mathcal{H}^{d-1}(J_{u_k}) + C_M \theta_l \le M + C_M\theta_l, \notag \\
(v) & \ \  \mathcal{H}^{d-1} \Big(\bigcup\nolimits_{j \ge 0} \partial^* P_j^{k,l}  \Big)   \le C_M.
\end{align}
\BBB By the fact that  $(|\nabla h_k|^p)_k$ are equi-integrable and \eqref{eq: comp1}(iii) \EEE we find a \UUU decreasing \EEE sequence $\eta_l \to 0$ as $l \to \infty$  such that 
 $$\Vert \nabla h_k \Vert^p_{L^p(R^l_k)} \le \eta_l   \ \ \text{ for all $k,l \in \N$}. $$
From \eqref{eq: small set main-boundary}(i) we thus obtain
 \begin{align}\label{eq: comp1-new}
E_k(v^l_k) \le E_k(u_k) +  C_M(\theta_l + \eta_l).
\end{align}  
For later purposes, we remark that for each collection $(\hat{t}^{k,l}_j)_{j \ge 1}$ with $|\hat{t}^{k,l}_j - t^{k,l}_j| \le \UUU \theta_l^{-1} \EEE \Vert v^l_k-h_k\Vert_\infty$ for all $j$, the functions $\hat{v}^l_k =  h_k \chi_{R^l_k} +  \sum\nolimits_{j \ge 1} (u_k - \hat{t}^{k,l}_j) \chi_{P_j^{k,l}}$ also satisfy 
\begin{align}\label{eq: comp1-new-strange}
E_k(\hat{v}^l_k) \le E_k(u_k) +  C_M(\theta_l + \eta_l), \ \ \ \ \ \ \ \ \mathcal{H}^{d-1}(J_{\hat{v}^l_k}) \le   \mathcal{H}^{d-1}(J_{u_k}) + C_M \theta_l,
\end{align}
see \eqref{eq: extra energy2}. \UUU We also observe that it is not restrictive to assume that
\begin{align}\label{eq: max3}
\Vert v^{l+1}_k - h_k \Vert_{L^\infty(\Omega')} \ge  \Vert v^{l}_k - h_k \Vert_{L^\infty(\Omega')}  \ \ \ \text{for all } \  l,k \in \N.   
\end{align}
In fact, otherwise we may replace the function $v^l_k$ defined  in \eqref{eq: comp4} for index $l$ by the function $v^{l+1}_k$. Then \eqref{eq: comp1}-\eqref{eq: comp1-new-strange} still hold as the sequences $\eta_l$  and $\theta_l$ are  decreasing and \eqref{eq: max3} is trivially satisfied. \EEE

%
%
%
%

\emph{Step 2 (Limiting objects for each $l$).} \BBB In view of \eqref{eq: comp1}(i),(ii),(iv) and the fact that $h_k$ converges to $h$ in $L^p(\Omega';\R^m)$, Ambrosio's compactness result (Theorem \ref{th: GSBD comp}) is applicable for fixed $l \in \N$. Thus, \EEE using a diagonal argument we get a  subsequence of $(k)_{k \in \N}$ (not relabeled) such that for every $l \in \N$ we find a function $v^l \in GSBV^p(\Omega';\R^m)$  with $v^l_k \to v^l$ pointwise a.e.\ in $\Omega'$ for $k\to \infty$. \RRR Clearly, by \eqref{eq: comp1}(i) we then also have
\begin{align}\label{eq: comp2}
v^l_k \to v^l \text{  in $L^1(\Omega';\R^m)$}.
\end{align}
Likewise, we can establish a compactness result for the Caccioppoli partitions \RRR as follows: in view of \eqref{eq: comp1}(iii), for a suitable subsequence of $(l)_{l \in \N}$ (not relabeled) we may suppose that 
\begin{align}\label{eq: comp1-XXXX-new}
\mathcal{L}^d(R^l_k) <\mathcal{L}^d(\Omega' \setminus \overline{\Omega})
\end{align}
 for all $k,l\in\N$.  Recall from Corollary \ref{cor: translations-bdy} that for each partition  at most one component $(P^{k,l}_j)_{j \ge 1}$ intersects $\Omega' \setminus \overline{\Omega}$. (We emphasize that the rest set $R^l_k$ is \emph{not}   counted among the components   here.)  In view of \eqref{eq: comp1-XXXX-new}, exactly one of these components intersects $\Omega' \setminus \overline{\Omega}$. We may reorder the components of the partitions such that $P_0^{k,l} = R^l_k$, such that
 \begin{align}\label{eq: boundary component}
\text{exactly $P^{k,l}_{1}$ intersects $\Omega' \setminus \overline{\Omega}$,}
\end{align}
 and $(P_j^{k,l})_{j \ge 2}$ are ordered for all $k, l \in \N$.  By  \eqref{eq: comp1}(v), Theorem \ref{th: comp cacciop}, and the comment thereafter we find for each $l \in \N$ a  partition $(P_j^l)_{j \ge 0}$ with  $\sum_{j\ge 0} \mathcal{H}^{d-1}(\partial^* P^{l}_j) \le C_M$ such that for a suitable subsequence of $(k)_{k \in \N}$ one has  $\sum_{j \ge 0} \mathcal L^d(P_j^{k,l} \triangle P_j^l)\to 0$ for $k \to \infty$. (Here, we again use a diagonal argument.) \EEE

Since $\sum_{j \ge 0} \mathcal{H}^{d-1}(\partial^* P^{l}_j) \le  C_M$ for all $l \in \N$, we can repeat the arguments and get a partition $(P_j)_{j \ge 0}$ such that $\sum_{j \ge 0} \mathcal L^d\left(P_j^{l} \triangle P_j\right)\to 0$ for $l \to \infty$ after extracting a further suitable subsequence. Thus, using a diagonal argument, we can choose a (not relabeled)  subsequence  of $(l)_{l \in \N}$ and afterwards of $(k)_{k \in \N}$ such that
\begin{align}\label{eq: comp1-XXXX}
\sum\nolimits_{j \ge 0} \mathcal L^d\big(P_j^{l} \triangle P_j\big)\le 2^{-l}, \ \ \ \ \ \ \sum\nolimits_{j \ge 0} \mathcal L^d\big(P_j^{k,l} \triangle P_j^l\big)\le 2^{-l} \ \ \text{ for all } k \ge l.
\end{align}

 \RRR Our goal is to obtain the desired function $\psi$ by using Lemma  \ref{rig-lemma: concave function2} for the limiting sequence $(v^l)_l$. We will now  show that, by redefining the translations on the  components of the partitions appropriately (cf. \eqref{eq: comp4}), we can indeed construct this sequence (which we will denote by $(\hat{v}^l)_l$ for better distinction) in such a way that \EEE  
\begin{align}\label{eq: comp5}
\mathcal L^d\left(\bigcap\nolimits_{n \in \N}\bigcup\nolimits_{ m \ge n}\lbrace |\hat{v}^n - \hat{v}^{m}| > 1\rbrace\right)=0.
\end{align}
Then Lemma  \ref{rig-lemma: concave function2} is applicable.

\emph{Step 3 (Redefinition of translations).}  We now come to the details how to choose the translations. Fix $k \in \N$. We describe an iterative procedure to redefine $t^{k,l}_j$ for all $l,j \in \N$. Let $\hat{v}^1_k = v^1_k$ as defined in \eqref{eq: comp4}. Assume that  $(\hat{t}^{k,l}_j)_j$ (which may differ from $(t^{k,l}_j)_j$)  and the corresponding $\hat{v}^l_k$ (see \eqref{eq: comp4}) have been chosen  such that 
\begin{align}\label{eq: comp1-NNNN}
\Vert \hat{v}^l_k - h_k \Vert_{L^\infty(\Omega')}  \le 2\sum\nolimits_{\ell =1}^{l} \Vert v^{\ell}_k - h_k \Vert_{L^\infty(\Omega')}
\end{align}
and $\hat{v}_k^l = h_k$ on $\Omega' \setminus \overline{\Omega}$. Clearly, these assumptions hold for $l=1$. 

Consider some $P^{k,l+1}_{j}$.  If $\mathcal{L}^d(P^{k,l}_{j} \cap P^{k,l+1}_j) > 0$, we define $\hat{t}^{k,l+1}_{j} = \hat{t}^{k,l}_{j}$. Otherwise, we  set $\hat{t}^{k,l+1}_{j} = t^{k,l+1}_{j}$. In the first case, noting that $v^{l+1}_k = u_k  -  t^{k,l+1}_{j}$ and $\hat{v}^l_k = u_k - \hat{t}^{k,l+1}_{j} $ on $P^{k,l}_{j} \cap P^{k,l+1}_j$ by  \eqref{eq: comp4},   we obtain by the triangle inequality and \eqref{eq: comp1-NNNN}
\begin{align}\label{eq: longt}
|\hat{t}^{k,l+1}_{j} -  t^{k,l+1}_{j}| &  = \Vert  u_k  -  t^{k,l+1}_{j} - (u_k - \hat{t}^{k,l+1}_{j} )  \Vert_{L^\infty(P^{k,l}_{j} \cap P^{k,l+1}_j)} = \Vert  v^{l+1}_k - \hat{v}^{l}_k  \Vert_{L^\infty(P^{k,l}_{j} \cap P^{k,l+1}_j)}\notag  \\
& \le  \Vert v^{l+1}_k - h_k \Vert_{L^\infty(\Omega')} + \Vert \hat{v}^l_k - h_k \Vert_{L^\infty(\Omega')}\notag \\
&  \le 2\sum\nolimits_{\ell =1}^{l} \Vert v^{\ell}_k - h_k \Vert_{L^\infty(\Omega')} + \Vert v^{l+1}_k - h_k \Vert_{L^\infty(\Omega')}.
\end{align}
We define $\hat{v}^{l+1}_k$ as in \eqref{eq: comp4} replacing $t_j^{k,l+1}$ by $\hat{t}_j^{k,l+1}$ and derive by the previous calculation 
$$\Vert \hat{v}_k^{l+1} -h_k\Vert_{L^\infty(\Omega')} \le  2\sum\nolimits_{\ell =1}^{l+1} \Vert v^{\ell}_k - h_k \Vert_{L^\infty(\Omega')},$$
i.e., \eqref{eq: comp1-NNNN} holds for $l+1$. We also have $\hat{v}^{l+1}_k = h_k$ on $\Omega' \setminus \overline{\Omega}$. In fact, by \eqref{eq: boundary component} only $P^{k,l+1}_{1}$ intersects $\Omega' \setminus \overline{\Omega}$. But then \eqref{eq: comp4} and $\BBB v_k^{l+1} = \EEE v_k^l = \hat{v}_k^l =u_k =h_k$ on $\Omega' \setminus \overline{\Omega}$ imply $t^{k,l+1}_{1} = \hat{t}^{k,l+1}_{1} = t^{k,l}_{1} = \UUU \hat{t}^{k,l}_{1} =0\EEE$ and thus $\hat{v}^{l+1}_k = h_k$ on $\Omega' \setminus \overline{\Omega}$.

By \EEE \eqref{eq: longt} and \eqref{eq: max3} we observe 
\begin{align*}
|\hat{t}^{k,l+1}_{j} -  t^{k,l+1}_{j}| &\le \UUU (2l +1 ) \EEE \Vert v^{l+1}_k - h_k \Vert_{L^\infty(\Omega')} \le \theta_{l+1}^{-1}  \Vert v^{l+1}_k - h_k \Vert_{L^\infty(\Omega')},
\end{align*}
\UUU where the last step follows from \eqref{eq: theta0}. \EEE Thus, in view of the remark before \eqref{eq: comp1-new-strange}, also the newly constructed functions $\hat{v}^{l+1}_k$ satisfy the energy bound \eqref{eq: comp1-new-strange}. 

 \BBB By \eqref{eq: comp1}(i) and \eqref{eq: comp1-NNNN} we also have $\Vert \hat{v}^l_k - h_k \Vert_{L^\infty(\Omega')}  \le 2\sum\nolimits_{\ell =1}^{l} C_{\theta_\ell, M}$. Thus, repeating the argument in  \eqref{eq: comp2}, we find some $\hat{v}^l \in GSBV^p(\Omega';\R^m)$  such that 
\begin{align}\label{eq: L5}
\hat{v}^l_k \to \hat{v}^l \text{  in $L^1(\Omega';\R^m)$}.
\end{align}

\emph{Step 4 (Proof of \eqref{eq: comp5}).} \BBB Having redefined the piecewise translated functions, we are  now in a position to \EEE show that \eqref{eq: comp5} holds. To this end, we set $A^n_{k,l} = \bigcap_{n \le m \le l} \lbrace \hat{v}^m_k =  \hat{v}^n_k \rbrace$ for all $n \in \N$ and  \BBB $k \ge l \ge n$. \EEE If we show
\begin{align}\label{eq: comp8}
\mathcal L^d\left(\Omega' \setminus A^n_{k,l}\right) \le c2^{-n}
\end{align}
for  $c=c(C_M)$,  then \eqref{eq: comp5} follows. In fact, for each $l\ge n$ we can choose $K=K(l)\ge l$ so large that $\mathcal L^d\left(\lbrace |\hat{v}^m_K - \hat{v}^m| > \frac{1}{2} \rbrace\right) \le 2^{-m}$ for all $n \le m \le l$ as $\hat{v}^m_k \to \hat{v}^m$ in measure for $k \to \infty$ \BBB (see \eqref{eq: L5}). \EEE This implies 
\begin{align*}
&\mathcal L^d\left(\bigcup\nolimits_{n \le m \le l} \lbrace |\hat{v}^m -\hat{v}^n|> 1 \rbrace\right)\le \mathcal L^d\left(\Omega' \setminus A^n_{K,l}\right) + \sum\nolimits_{n \le m \le l} \mathcal L^d\left(\lbrace |\hat{v}^m_K - \hat{v}^m| >  \tfrac{1}{2} \rbrace\right) \le c2^{-n}.
\end{align*}
(Here, the constant $c$ may vary from step to step.) \RRR Passing to the limit $l \to \infty$ we find $\mathcal L^d(\bigcup\nolimits_{n \le m} \lbrace |\hat{v}^m -\hat{v}^n|> 1 \rbrace) \le c2^{-n}$ and taking the intersection over all $n \in \N$ we obtain \eqref{eq: comp5}, \EEE as desired.

We now show \eqref{eq: comp8}.  First, observe that by \eqref{eq: comp4}, \eqref{eq: comp1}(iii), and \BBB $\theta_m \le 2^{-m}$ (see \eqref{eq: theta0}) \EEE
\begin{align}\label{lll}
\mathcal{L}^2\Big( \bigcap\nolimits_{n \le m \le l} \lbrace T^n_k = T^m_k \rbrace \setminus A^n_{k,l}  \Big) \le \sum_{n \le m \le l} \mathcal{L}^d(P^{k,m}_0) = \sum_{n \le m \le l} \mathcal{L}^d(R^m_k)  \le C_M2^{-n},
\end{align}
where  $T^m_k := \sum_{j \ge 1} \hat{t}_j^{k,m}\chi_{P_j^{k,m}}$ and \BBB $P_0^{k,m} = R^m_k$. \EEE     Due to the above construction of the translations in Step 3, we get $\lbrace T^m_k = T^{m+1}_k \rbrace \supset \bigcup_{j \ge 1} (P^{k,m+1}_j \cap P^{k,m}_j)$ for $n \le m \le l-1$.  From \eqref{eq: comp1-XXXX} we deduce $\sum_{  j \ge 0  } \mathcal L^d(P_j^{k,m+1} \triangle P_j^{k,m}) \le 3 \cdot 2^{-m}$.  This along with \BBB \eqref{eq: theta0} and \EEE \eqref{eq: comp1}(iii) yields     
\begin{align*}
\mathcal L^d\left(\Omega' \setminus \lbrace T^m_k = T^{m+1}_k \rbrace\right) & \le \mathcal{L}^d(R^{m+1}_k) +  \sum\nolimits _{j \ge 1} \mathcal L^d\big(P_j^{k, m+1} \setminus P_j^{k, m}\big)   \le C_M 2^{-(m+1)} + 3 \cdot 2^{-m}.
\end{align*}
We now sum over $n \le m \le l-1$ and, in view of \eqref{lll}, we obtain \eqref{eq: comp8}. Thus, as already shown above,  also \eqref{eq: comp5} holds.

\emph{Step 5 (Conclusion).}  \BBB We observe that $(\hat{v}^l)_l \subset L^1(\Omega';\R^m)$ by \eqref{eq: L5}.  In view of \eqref{eq: comp5},   we can apply Lemma \ref{rig-lemma: concave function2} to obtain a nonnegative, increasing, concave function $\tilde{\psi}$ with 
$\lim_{t\to+\infty}\tilde{\psi}(t)=+\infty$ such that (up to a subsequence)
\begin{align}\label{eq: sup control}
\sup\nolimits_{l \ge 1} \int_{\Omega'} \tilde{\psi}(|\hat{v}^l|)\, dx < + \infty.
\end{align} 
Define $\psi(t) =  \min \lbrace \tilde{\psi}(t) ,t \rbrace$ and observe that $\psi$ has the properties stated in Theorem \ref{th: comp-aux}. \EEE  We are now in a position to define the modifications $(y_k)_k$ with the desired properties. Recalling $\hat{v}_k^l \to \hat{v}^l$ in $L^1(\Omega';\R^m)$ (see   \eqref{eq: L5}) and \eqref{eq: comp1-new-strange}, we can select a subsequence of  $(u_k)_k$ and a diagonal sequence $(y_k)_k \subset (\hat{v}^l_k)_{k,l}$ such that $\Vert y_k - \hat{v}^l \Vert_{L^1(\Omega')} \le 1$ for some $\hat{v}^l$, 
$$E_k(y_k) \le E_k(u_k) + \tfrac{1}{k}, \ \ \  \ \ \ \ \ \  \mathcal{H}^{d-1}(J_{y_k}) \le \mathcal{H}^{d-1}(J_{u_k}) + \tfrac{1}{k},$$ 
and $y_k = h_k$ on $\Omega' \setminus \overline{\Omega}$. This yields \eqref{eq: compi1-aux}. In fact, (i) follows from the previous equation and (iii) follows from  \eqref{eq: comp4} and \eqref{eq: comp1}(iii). Finally, to see (ii), we observe that $\psi$ is subadditive as concave function with $\psi(0) \ge 0$. Then  $\Vert y_k - \hat{v}^l \Vert_{L^1(\Omega')} \le 1$ implies
$$\sup\nolimits_{k \in \N} \int_{\Omega'} {\psi}(|y_k|)\, dx  \le \sup\nolimits_{l \ge 1} \Big( \int_{\Omega'} \tilde{\psi}(|\hat{v}^l|)\, dx + \Vert y_k - \hat{v}^l \Vert_{L^1(\Omega')} \Big) \le \sup\nolimits_{l \ge 1} \int_{\Omega'} \tilde{\psi}(|\hat{v}^l|)\, dx + 1. $$
By \eqref{eq: sup control}  this concludes the proof.
\end{proof}

\begin{remark}\label{rem: leichter}
{\normalfont

\RRR We close this section with the observation that Theorem \ref{th: comp-aux} is  much easier to prove if $(g3)$ is replaced by a condition of the form 
\begin{align}\label{eq: newcond}
c_4 \big(1+\varphi(|\zeta|) \big) \le  g(x,\zeta,\nu) \ \ \ \ \text{ for every $x \in \Omega'$, $\zeta \in \R^m_0$, and  $\nu \in \mathbb{S}^{d-1}$},
\end{align}
where $\varphi: \R_+ \to \R_+$ is an increasing function satisfying $\varphi(t) \le t$ for all $t \in \R_+$  and $\lim_{t \to \infty} \varphi(t) = +\infty$. Indeed, in this case no modifications have to be introduced, but \eqref{eq: compi1-aux}(ii) can be shown  for the original sequence $(u_k)_k$. The strategy is to apply the (standard) Poincar\'e inequality in $BV$ on a suitable composition of $u_k$ with some $\psi$, which allows to  control uniformly the $L^1$-norm of  the compositions and leads to  \eqref{eq: compi1-aux}(ii). \EEE

Let us come to the details.  Consider a sequence $(E_k)_k \subset \mathcal{E}_{\Omega'}$  with densities $f_k$ and $g_k$
and $(u_k)_k  \subset GSBV^p(\Omega';\R^m)$ with   $\sup_{k \in \N} E_k(u_k) \le \BBB C_* \EEE <+\infty$. As ${\varphi}$ is increasing and satisfies $\lim_{t \to \infty} {\varphi}(t) = + \infty$,  we can find a smooth,  increasing, concave function $\psi: \R_+ \to \R_+$ with \RRR $\psi \le {\varphi}+2$ \EEE and $\lim_{t \to \infty} \psi(t) = + \infty$.
(An elementary construction of such a function may be found in \cite[Lemma 4.1]{Friedrich:15-2} using an increasing sequence $(b_i)_i$ satisfying ${\varphi}(b_i) \ge 2^{i+1}$ for $i \in \N$.)  Observe that, as concave function with $\psi(0) \ge 0$, $\psi$ is subadditive.
Our goal is to show that for each   $i=1,\ldots,m$ we have
\begin{align}\label{eq: psi component}
\sup\nolimits_{k \in \N} \int_{\Omega'} \psi (|u^i_k| ) \,dx \le C < + \infty  \ \ \ \text{for all $k \in \N$}.
 \end{align}
Here and in the following, the superscript indicates the $i$-th component. Once \eqref{eq: psi component} is established, we can conclude  $\int_{\Omega'} \psi (|u_k| ) \,dx  \le \sum_{i=1}^m \int_{\Omega'} \psi (|u^i_k| ) \,dx \le Cm$ by the subadditivity of $\psi$.

Let us now confirm \eqref{eq: psi component}. We define the function $v^i_k = \psi(|u^i_k|) \in GSBV(\Omega')$ and note that \BBB  $|\nabla v^i_k | \le \Vert \psi' \Vert_\infty |\nabla u^i_k| \le \Vert \psi' \Vert_\infty |\nabla u_k|$ $\mathcal{L}^d$-a.e.\ in $\Omega'$. By ($f2$) and $E_k(u_k) \le \BBB C_* \EEE$ this implies  
\begin{align}\label{eq: casea1}
\Vert \nabla v^i_k \Vert_{L^p(\Omega')}^p \le \Vert \psi' \Vert_\infty^p\Vert \nabla u_k \Vert_{L^p(\Omega')}^p \le \Vert \psi' \Vert_\infty^p c_1^{-1} \int_{\Omega'}  f_k(x,\nabla u_k(x))\, dx \le \Vert \psi' \Vert_\infty^p c_1^{-1} \BBB C_*. \EEE
\end{align}
Moreover, for  $\mathcal{H}^{d-1}$-a.e.\ \UUU point of \EEE $J_{v^i_k}$ we find by the fact that $\psi$ is increasing and subadditive 
$$|[v^i_k]| = |(v^i_k)^+ - (v^i_k)^-| = | \psi(|(u^i_k)^+|) -\psi(|(u^i_k)^-|)| \le \psi(|(u^i_k)^+ - (u^i_k)^-|) = \psi(|[u^i_k]|)  \le \psi(|[u_k]|).  $$
Using $\psi \le \varphi +2$, \eqref{eq: newcond}, and $E_k(u_k) \le \BBB C_* \EEE$ we derive
\begin{align}\label{eq: casea2}
\int_{J_{v^i_k}}  |[v^i_k]| \, d\mathcal{H}^{d-1} \le \int_{J_{u_k}} \big( \varphi(|[u_k]|)+2\big) \, d\mathcal{H}^{d-1} \le \frac{2}{c_4} \int_{J_{u_k}}  g_k(x, [u_k], \nu_{u_k} ) \, d\mathcal{H}^{d-1} \le \frac{2}{c_4} \BBB C_*. \EEE 
\end{align}
Now H\"older's inequality and \eqref{eq: casea1}-\eqref{eq: casea2} imply the bound  $|D v^i_k|(\Omega') \le C$  on the total variation, where $C=C(C_*, \Omega', c_1, c_4,p, \Vert \psi' \Vert_\infty)$.  By the Poincar\'e inequality in $BV$ (see \cite[Remark 3.50]{Ambrosio-Fusco-Pallara:2000}) we therefore find $b^i_k \in \R$ such that
$$\Vert v^i_k  - b^i_k \Vert_{L^1(\Omega')} \le C|D v^i_k|(\Omega') \le C.$$
As $v^i_k = \psi(|h^i_k|) $ on $\Omega' \setminus \overline{\Omega}$, we also deduce $\Vert \psi(|h^i_k|)  - b^i_k \Vert_{L^1(\Omega' \setminus \Omega)} \le C$ and therefore
 $$\Vert v^i_k  \Vert_{L^1(\Omega')} \le   C + C\Vert \psi(|h_k|) \Vert_{L^1(\Omega')}.$$
Using $\psi(t) \le \varphi(t) + 2 \le t+2$, we note that $\Vert \psi(|h_k|) \Vert_{L^1(\Omega')}$ is uniformly bounded in $k$. Recalling $v^i_k = \psi(|u^i_k|)$, this shows \eqref{eq: psi component} and concludes the proof. \nopagebreak\hspace*{\fill}$\Box$
}
 \end{remark}

\section{Existence and $\Gamma$-convergence results for free discontinuity problems}\label{sec: appli}

In this section we provide some applications of the compactness result to boundary value problems. In the following, we suppose that there exist two bounded Lipschitz domains $\Omega' \supset \Omega$. We will impose Dirichlet boundary data on  $\partial_D \Omega := \Omega' \cap \partial \Omega$. As usual for the weak formulation in the frame of $SBV$ functions, this will be done by requiring that configurations $u$ satisfy $u = h$ on $\Omega' \setminus \overline{\Omega}$ for some $h \in W^{1,p}(\Omega';\R^m)$. We will first present an existence result and then address $\Gamma$-convergence for energies in the class $\mathcal{E}_\Omega$.

\subsection{Existence}
As a first application, we prove an existence result for energy functionals in the class $\mathcal{E}_{\Omega'}$ introduced in Section \ref{sec: formulation}.

\begin{theorem}[Existence result for free discontinuity problems in $GSBV^p$]\label{th: existence}
Let $\Omega \subset \Omega' \subset \R^d$ be bounded Lipschitz domains. Let $E \in \mathcal{E}_{\Omega'}$ be lower semicontinuous  in $L^0(\Omega';\R^m)$ and let $h \in  W^{1,p}(\Omega';\R^m)$. Then the minimization problem 
\begin{align*}
\inf_{{u \in \UUU L^0(\Omega';\R^m)\EEE}} \lbrace E(u): \  u = h \text{ on } \Omega' \setminus \overline{\Omega} \rbrace
\end{align*}
admits solutions. 
\end{theorem}
 
\begin{proof}
The result follows from Theorem \ref{th: comp} and the direct method. Indeed, choosing a  minimizing   sequence $(u_k)_k$, we find another minimizing sequence $(y_k)_k$ converging in measure to some $u \in GSBV^p(\Omega';\R^m)$ with $u = h$ on $\Omega' \setminus \overline{\Omega}$. The lower semicontinuity of $E$ with respect to convergence in measure then yields that $u$ is a minimizer.  
\end{proof}

Without going into details, let us just briefly mention that in  \cite{Ambrosio:90, Ambrosio:90-2}, lower semicontinuity for functionals $E \in \mathcal{E}_{\Omega'}$ with respect to measure convergence  is ensured (under the assumption that $g$ is continuous) by quasiconvexity for the bulk density $f$ and $BV$-ellipticity  \cite{AmbrosioBraides2} for the surface density $g$.

Clearly, the minimizer of the problem is independent of the definition of $f(x,\xi)$ for $x \in \Omega' \setminus {\Omega}$ and independent of $g(x,\zeta,\nu)$ for \BBB $x \in \Omega' \setminus \overline{\Omega}$. \EEE The value of  $g(x,\zeta,\nu)$ for $x \in \partial_D \Omega$, however, may affect the minimization problem. Indeed, it might be energetically favorable if the crack runs alongside $\partial_D \Omega$. In this case, the boundary datum is not attained in the sense of traces, at the expense of a crack energy. Below in Section \ref{sec: gamma}, we will present a variant where the minimizer is determined only by $g(x,\zeta,\nu)$ for $x \in \Omega$,   see Remark \ref{rem: bdy2}.

\subsection{$\Gamma$-convergence}\label{sec: gamma}

We now revisit the $\Gamma$-convergence result for free discontinuity problems established recently in \cite{Caterina}. There, for  minimization problems involving an $L^p$-perturbation of the  energy functionals \eqref{eq: main energy functional},   convergence of minimum values and minimizers is proved. In the present contribution, we treat boundary value problems without any $L^p$-perturbation instead. 
 
 For the application to $\Gamma$-convergence results,  we need some further assumptions on the bulk density $f: \Omega \times \R^{m \times d} \to \R_+$ and the surface density $g: \Omega \times \R^m_0 \times \mathbb{S}^{d-1} \to \R_+$, see \cite{Caterina}. Let $c_1,\ldots,c_5$ be the constants in the definition of $\mathcal{E}_\Omega$ in Section \ref{sec: compactness}. Moreover, we let 
$\sigma_1, \sigma_2: \R_+ \to \R_+$ be two nondecreasing continuous functions with $\sigma_1(0) = \sigma_2(0) =0$. By $\mathcal{E}'_\Omega \subset \mathcal{E}_\Omega$ we denote  the collection of   integral functionals \eqref{eq: main energy functional} where additionally the following holds:
\begin{itemize}
\item[($f3$)] (continuity in $\xi$) for every $x \in \Omega$   we have
$$|f(x,\xi_1) - f(x,\xi_2) | \le \sigma_1(|\xi_1 - \xi_2|)(1+f(x,\xi_1) + f(x,\xi_2)) $$
for every $\xi_1,\xi_2 \in \R^{m\times d}$,
\item[($g5$)] (estimate for $|\zeta_1| \le |\zeta_2|$) for every $x \in \Omega$ and every $\nu \in \mathbb{S}^{d-1}$ we have
$$g(x, \zeta_1 , \nu)   \le c_3 g(x, \zeta_2, \nu)$$
for every $\zeta_1, \zeta_2 \in \R^m_0$ with $|\zeta_1| \le |\zeta_2|$,
\item[($g6$)] (continuity in $\zeta$) for every $x \in \Omega$ and every $\nu \in \mathbb{S}^{d-1}$ we have
$$|g(x,\zeta_1,\nu) - g(x,\zeta_2,\nu) | \le \sigma_2(|\zeta_1 - \zeta_2|)(g(x,\zeta_1,\nu) + g(x,\zeta_2,\nu)) $$
for every $\zeta_1,\zeta_2 \in \R^m_0$.
\end{itemize}
Besides the two continuity conditions, in \cite{Caterina} additionally  ($g5$) is required which is a kind of `monotonicity condition' for the jump height $|\zeta|$. We refer to \cite[Remark 3.2, 3.3]{Caterina} for more details. We include ($g5$) here only for the reader's convenience to ease reference to the assumptions in \cite{Caterina}. Actually, the condition already follows (with different constants) from ($g3$). In the following we denote by $\mathcal{A}(\Omega)$ the open subsets of $\Omega$.


\begin{theorem}[Compactness of $\Gamma$-convergence, see \cite{Caterina}]\label{th: caterina}
Let $(E_k)_k$ be a sequence in $\mathcal{E}'_\Omega$ with densities $(f_k)_k$ and $(g_k)_k$. Then there exists a  subsequence (not relabeled) and  $f: \Omega \times \R^{m \times d} \to \R_+$,  $g: \Omega \times \R^m_0 \times \mathbb{S}^{d-1} \to \R_+$ such that for all $A \in \mathcal{A}(\Omega)$
$$E_k(\cdot,A) \text{ $\Gamma$-converges to } E(\cdot, A) \text{ in $L^0(\Omega;\R^m)$}, $$
where $E: L^0(\Omega;\R^m) \times \mathcal{A}(\Omega) \to [0,+\infty]$ is given by \eqref{eq: main energy functional} and lies in $\mathcal{E}_{\Omega}'$. Moreover, we have
$$E^p_k(\cdot,A) \text{ $\Gamma$-converges to } E^p(\cdot, A) \text{ in $L^p(\Omega;\R^m)$}, $$
where $E^p_k$ and $E^p$ denote the restriction of $E_k$ and $E$ to $L^p(\Omega;\R^m) \times \mathcal{A}(\Omega)$, respectively. 
\end{theorem}

For a general theory of $\Gamma$-convergence we refer the reader to \cite{DalMaso:93}. The limiting bulk density $f$ and surface density $g$ associated to $E$ can be expressed in terms of the densities $f_k$ and $g_k$ via specific asymptotic cell formulas, see \cite[Theorem 3.5, Theorem 5.2]{Caterina}.  The crucial point is that the
problems for the volume and surface integrals are decoupled, i.e., $f$ depends only on the sequence $(f_k)_k$
while $g$ depends only on the sequence $(g_k)_k$. In particular, for $A \in \mathcal{A}(\Omega)$ and a sequence $(u_k)_k$ with $\sup_{k \in \N} E_k(u_k,A) < + \infty$ converging to  $u$ in measure on $A$, we have
\begin{align}\label{eq: separation}
\int_A f(x, \nabla u(x)) \, dx &\le  \liminf_{k \to \infty} \int_A f_k(x, \nabla u_k(x)) \, dx, \notag \\ \int_{J_u \cap A} g(x, [u],\nu_u) \, d\mathcal{H}^{d-1} & \le  \liminf_{k \to \infty} \int_{J_{u_k} \cap A} g_k(x, [u_k],\nu_{u_k}) \, d\mathcal{H}^{d-1}.
\end{align} 
A proof of this fact may be found in \cite[Proposition 4.3]{Giacomini-Ponsiglione:2006}. Now suppose that $(u_k)_k$ is a recovery sequence for $u$ with respect to the $L^p(\Omega;\R^m)$-convergence. We will use  the following general fact several times: if \UUU $A \in \mathcal{A}(\Omega)$  \EEE with $E(u,\partial A) = 0$, then $(u_k)_k$ is  also a recovery sequence with respect to $E_k(\cdot,A)$, see \cite[Remark 3.6]{Giacomini-Ponsiglione:2006}. Thus, if $E(u,\partial A) = 0$, we find by \eqref{eq: separation}
\begin{align}\label{eq: separation-recovery}
\int_A f(x, \nabla u(x)) \, dx & =  \lim_{k \to \infty} \int_A f_k(x, \nabla u_k(x)) \, dx, \notag \\ \int_{J_u \cap A} g(x, [u],\nu_u) \, d\mathcal{H}^{d-1} & =  \lim_{k \to \infty} \int_{J_{u_k} \cap A} g_k(x, [u_k],\nu_{u_k}) \, d\mathcal{H}^{d-1}.
\end{align}
\UUU Consider again bounded Lipschitz domains $\Omega' \supset \Omega$ and suppose that also $\Omega' \setminus \overline{\Omega}$ has Lipschitz boundary. \EEE To treat non-attainment of the boundary data (in the sense of traces) as internal jumps, we introduce energy functionals defined on $\Omega'$. We set
\begin{align}\label{eq: f ext}
f'_k(x,\xi) := \begin{cases} f_k(x,\xi) & \text{if } x \in \Omega, \\ c_1 |\xi|^p&  \text{otherwise.} \end{cases}
\end{align}
and 
\begin{align}\label{eq: g ext}
g'_k(x,\zeta,\nu) := \begin{cases} g_k(x,\zeta,\nu) & \text{if } x \in \Omega, \\ c_5 + 1 &  \text{otherwise.} \end{cases}
\end{align}

According to Theorem \ref{th: caterina}, the functionals $E_k' \in \mathcal{E}_{\Omega'}'$, with densities $f_k'$ and $g_k'$, $\Gamma$-converge in \UUU $L^0(\Omega';\R^m)$ \EEE (up to a subsequence) to some $E' \in \mathcal{E}_{\Omega'}'$ with densities $f'$ and $g'$. (Strictly speaking, we consider here the class $\mathcal{E}_{\Omega'}'$ with $c_5+1$ instead of $c_5$.) Then we clearly have 
$$
 f'(x,\xi) = \begin{cases} f(x,\xi) & \text{if } x \in \Omega, \\ c_1 |\xi|^p&  \text{otherwise} \end{cases}
$$
and $g'(x,\zeta,\nu) = g (x,\zeta,\nu)$ for $ x \in \Omega$. Below in Remark \ref{rem: bdy},  we will see that $g'(x,\zeta,\nu)$ for $x \in \partial_D \Omega$ is completely determined by the sequence $(g_k)_k$.

We now prove the following version of the $\Gamma$-convergence result that takes  boundary data into account.

\begin{lemma}[$\Gamma$-convergence with boundary data]\label{lemma: gamma bdy}
Suppose that $\Omega' \setminus \overline{\Omega}$ \BBB has Lipschitz boundary. \EEE  Let the sequence of functionals $(E'_k)_k \subset \mathcal{E}'_{\Omega'}$ with  densities $(f'_k)_k$, $(g'_k)_k$ and the limiting functional $E' \in \mathcal{E}'_{\Omega'}$ with  densities $f'$, $g'$ be given as above. Suppose that $(h_k)_k \subset W^{1,p}(\Omega';\R^m)$ converges strongly to $h$ in $W^{1,p}(\Omega';\R^m)$. Then the sequence of functionals 
$$\tilde{E}'_k(u) = \begin{cases} E'_k(u) & \text{if } u = h_k \text{ on } \Omega' \setminus \overline{\Omega}, \\ +\infty &  \text{otherwise},  \end{cases} $$ 
$\Gamma$-converges in $L^0(\Omega';\R^m)$ to 
$$\tilde{E}'(u) = \begin{cases} E'(u) & \text{if } u = h \text{ on } \Omega' \setminus \overline{\Omega}, \\ +\infty &  \text{otherwise.}  \end{cases}  $$

\end{lemma}

\begin{proof}
We follow the proof in \cite[Lemma 7.1]{Giacomini-Ponsiglione:2006}. In particular, we  highlight the necessary adaptions in our setting which are related to the fact that (a) the surface densities also depend on the crack opening and (b) we prove that   \UUU $g'$ \EEE is determined completely by $(g_k)_k$, see Remark \ref{rem: bdy}.

 First, the $\Gamma$-liminf is immediate from the $\Gamma$-convergence of $E'_k$ to $E'$ and the fact that the constraint is closed under the convergence in measure. We now address the  $\Gamma$-limsup.  Due to a general approximation argument in the theory of $\Gamma$-convergence together with Corollary \ref{cor: truncations}, it suffices to construct recovery sequences for $u \in GSBV^p(\Omega';\R^m) \cap L^p(\Omega';\R^m)$ with $u=h$ on $\Omega' \setminus \overline{\Omega}$.

By Theorem \ref{th: caterina} there exists a recovery sequence $(u_k)_k$ for $u$ with respect to $L^p$-convergence, i.e., $\Vert u_k - u \Vert_{L^p(\Omega')} \to 0$ and $\lim_{k\to \infty} E'_k(u_k) = E'(u)$. We note that \eqref{eq: separation}-\eqref{eq: separation-recovery} hold (for the densities defined in \eqref{eq: f ext}-\eqref{eq: g ext}). We claim that 
\begin{align}\label{eq: good approx}
(i) & \ \ u_k - h_k \to 0 \text{ strongly in } \UUU L^p(\Omega' \setminus \overline{\Omega};\R^{m}), \EEE \notag \\
(ii) & \ \ \nabla u_k - \nabla h_k \to 0 \text{ strongly in } L^p(\Omega' \setminus \overline{\Omega};\R^{m \times d}), \notag\\
(iii)& \ \  \mathcal{H}^{d-1}(J_{u_k} \cap (\Omega' \setminus {\Omega}) ) \to 0. 
\end{align}
We defer the proof of \BBB these properties \EEE to the end of the proof. 

\emph{Definition of the recovery sequence:} We can find a neighborhood $U \supset\supset \Omega' \setminus \overline{\Omega}$ and an extension $(y_k)_k\subset GSBV^p(U;\R^m)$ satisfying $y_k =u_k - h_k$ on $\Omega' \setminus \overline{\Omega}$ such that in view of \eqref{eq: good approx}
\begin{align}\label{eq: good approx2}
\Vert   y_k \Vert_{L^p(U)} + \Vert \nabla y_k \Vert_{L^p(U)} + \mathcal{H}^{d-1}(J_{y_k} \cap U) \to 0
\end{align}
as $k \to \infty$. This can be done, e.g., as in \cite[Theorem 3.1, Theorem 8.1]{Scardia} with $GSBV^p \cap L^p$ in place of $SBV^2 \cap L^2$. (In both cases, the problem can be reduced to more regular functions by approximation \cite{Cortesani:1997}.)

Let $\eps > 0$ and choose $V$ open with $V \supset \overline{\partial_D\Omega}$, $V \subset U$, $E'(u,\partial (V \cap \Omega')) = 0$, $\mathcal{L}^d(V)\le \eps$, and $\int_{V\cap \Omega'} f'(x,\nabla u(x))\, dx< \eps$. Then by \eqref{eq: separation-recovery} we also get 
\begin{align}\label{eq: limiting eps}
\limsup_{k \to \infty} \int_{V \cap \Omega'} f'_k(x,\nabla u_k(x)) \, dx < \eps.
\end{align}
Choose $W \subset \R^d$ open such that ${\Omega'} \setminus \overline{\Omega} \subset W$ and $\overline{W} \cap \overline{\Omega \setminus V} = \emptyset$.  Let $\psi \in C^\infty(\Omega')$ with \UUU $0 \le \psi \le 1$, \EEE $\psi = 0$ on $\Omega \setminus V$ and $\psi =1$ on $W \cap \Omega'$. Define $\varphi_k \in GSBV^p(\Omega';\R^m)$ by $\varphi_k =  \psi y_k$ on \UUU $U \cap \Omega'$ \EEE and $\varphi_k = 0$ else. Note   by \eqref{eq: good approx2} that
\begin{align}\label{eq: last g2}
\Vert   \varphi_k \Vert_{L^p(\Omega')} + \Vert \nabla \varphi_k \Vert_{L^p(\Omega')} + \mathcal{H}^{d-1}(J_{\varphi_k} \cap \Omega') \to 0.
\end{align}
Now we set $\tilde{u}_k  := u_k - \varphi_k$. Then $\tilde{u}_k = u_k - y_k$ on $W \cap \Omega'$ and thus $\tilde{u}_k = h_k$ on $\Omega' \setminus \overline{\Omega}$. Moreover, $\tilde{u}_k = u_k$ on $\Omega \setminus V$. We also observe that $\tilde{u}_k \to u$ in $L^p(\Omega';\R^m)$ by \eqref{eq: last g2}. We now estimate $\tilde{E}_k'(\tilde{u}_k)$. As $\mathcal{H}^{d-1}(J_{\varphi_k} \cap \Omega') \to 0$, we find by ($g3$)
\begin{align}\label{eq: last g}
\limsup_{k \to \infty} \int_{J_{\tilde{u}_k}} g'_k(x,[\tilde{u}_k],\nu_{\tilde{u}_k}) \, \mathcal{H}^{d-1} \le \limsup_{k \to \infty} \int_{J_{{u}_k}} g'_k(x,[{u}_k],\nu_{{u}_k}) \, \mathcal{H}^{d-1}.  
\end{align}
Moreover,  \eqref{eq: f ext} implies
\begin{align*}
\int_{\Omega'} |f'_k(x,\nabla u_k) - f'_k(x,\nabla \tilde{u}_k) |\, dx \le  \int_{V \cap \Omega} \big(f_k(x,\nabla u_k) + f_k(x,\nabla \tilde{u}_k) \big)  + c_1 \int_{\Omega' \setminus \Omega} ||\nabla u_k|^p -  |\nabla h_k|^p|.
\end{align*}
The rightmost term converges to zero for $k \to \infty$ by \eqref{eq: good approx}(ii). By using the growth conditions ($f2$), \eqref{eq: limiting eps}-\eqref{eq: last g2},  and $\mathcal{L}^d(V)\le \eps$   we find
 \begin{align*}
\limsup_{k \to \infty}  \int_{V \cap \Omega} \big(f_k(x,\nabla u_k) + f_k(x,\nabla \tilde{u}_k)\big)  \, dx  &\le c_2 \mathcal{L}^d( V) +  2^{p-1}c_2 \limsup_{k \to \infty}  \int_{V\cap \Omega} |\nabla \varphi_k|^p \, dx \\ &  \ \ \ + (1+2^{p-1}c_2 c_1^{-1}) \limsup_{k \to \infty}  \int_{V\cap \Omega} f_k(x,\nabla u_k) \, dx \\
& \le c_2\eps + (1+2^{p-1}c_2 c_1^{-1})\eps.
  \end{align*}
By   \eqref{eq: last g} \UUU and the fact that $E_k'(u_k) \to E'(u) = \tilde{E}'(u)$, \EEE we then derive
$$ \limsup_{k \to \infty}  \tilde{E}'_k(\tilde{u}_k) \le \limsup_{k \to \infty} E'_k( {u}_k)  +  c_2\eps + (1+2^{p-1}c_2 c_1^{-1})\eps \le \tilde{E}'(u) +   c_2\eps + (1+2^{p-1}c_2 c_1^{-1})\eps. $$  
 Since $\eps$ was arbitrary, using a diagonal argument we have proved the $\Gamma$-limsup inequality.

\emph{Proof of \eqref{eq: good approx}:}  To conclude, it remains to show \eqref{eq: good approx}. First, to see (i), we recall $u_k \to h$ in \UUU $L^p(\Omega' \setminus \overline{\Omega};\R^m)$ \EEE as $(u_k)_k$ is a recovery sequence in $L^p$. Then it suffices to use that $h_k \to h$ in $L^p(\Omega';\R^m)$. We now address (ii). Let $A \in \mathcal{A}(\Omega')$, $\overline{A} \subset \Omega' \setminus \overline{\Omega}$ with $E'(u,\partial A) = 0$.   Then \eqref{eq: separation-recovery} and \eqref{eq: f ext} imply 
\begin{align}\label{eq: first convergence}
\nabla u_k \to \nabla h \ \ \text{ in } \ L^p(A;\R^{m \times d}).
\end{align} For $\eps >0$  consider $V$ open with $V \supset \overline{\partial_D\Omega}$ such that  $E'(u,\partial (V\cap \Omega')) = 0$, $\mathcal{L}^d(V)< \eps$,  and 
$$\int_{V \cap \Omega'} f'(x,\nabla u(x)) \, dx < \eps, \ \ \ \ \int_{V \cap \Omega'} f'(x,\nabla h_k(x))\,dx <\eps \ \ \ \ \text{for all } k \in \N. $$
(The latter is possible by ($f2$) and the fact that $\nabla h_k \to \nabla h$ strongly in $L^p(\Omega';\R^{m\times d})$.)  
For $k$ large enough, we also have  $\int_{V \cap \Omega'} f'_k(x,\nabla u_k(x)) \, dx < \eps$ by \eqref{eq: separation-recovery}. Then we calculate by ($f2$)
\begin{align*}
\int_{\Omega' \setminus \overline{\Omega}} |\nabla u_k - \nabla h_k|^p\,dx    & = \int_{\Omega' \setminus (\Omega \cup V)}|\nabla u_k - \nabla h_k|^p\,dx    + \int_{V \cap (\Omega' \setminus \overline{\Omega})}|\nabla u_k - \nabla h_k|^p \,dx   \\
& \le  \int_{\Omega' \setminus (\Omega \cup V)}|\nabla u_k - \nabla h_k|^p\, dx   + \frac{2^{p-1}}{c_1}    \int_{V \cap \Omega'}\hspace{-0.2cm} (f'_k(x,\nabla u_k) +  f'(x,\nabla h_k))\,dx.
\end{align*}
Then \eqref{eq: first convergence} and the fact that $\Vert \nabla h_k - \nabla h \Vert_{L^p(\Omega')} \to 0$ yield 
$$\limsup_{k \to \infty}\int_{\Omega' \setminus \overline{\Omega}} |\nabla u_k - \nabla h_k|^p \, dx \le 2^p c_1^{-1} \eps.$$ Since $\eps$ was arbitrary, we obtain (ii). We finally prove (iii). Up to a subsequence we have
$$\mu_k := \mathcal{H}^{d-1}|_{J_{u_k} \cap (\Omega' \setminus {\Omega})} \stackrel{*}{\rightharpoonup} \mu  \ \ \ \text{weakly$^*$ in } \ \mathcal{M}_b(\Omega').$$
By \eqref{eq: separation-recovery} we observe $\mathcal{H}^{d-1}(J_{u_k} \cap U) \to 0$ for all $U \in \mathcal{A}(\Omega')$, $\overline{U} \subset \Omega' \setminus \overline{\Omega}$, and $E'(u,\partial U) = 0$. Consequently, to conclude the proof of (iii), it suffices to show $\mu(\partial_D\Omega) = 0$. We argue by contradiction. Let us assume that $\mu(\partial_D\Omega) > 0$. Then there exists a cube $Q_\rho$ with center $x \in \partial_D\Omega$ and sidelength $2\rho$ such that $Q_\rho \subset \Omega'$, \BBB $E'(u,\partial Q_\rho) = 0$, \EEE and   $\mu(Q_\rho) > \sigma > 0$. We may also suppose that $Q_{4\rho} \subset \Omega'$, where $Q_{4\rho}$  \UUU denotes the  cube with center $x$ and sidelength $8\rho$. \EEE For $k$ large enough we also have
\begin{align}\label{eq: sigma}
\mathcal{H}^{d-1}( J_{u_k} \cap (Q_\rho \setminus {\Omega}) )  = \mu_k(Q_\rho ) > \sigma >0. 
\end{align}

Following the proof of \cite[Lemma 7.1]{Giacomini-Ponsiglione:2006},  one can modify the sequence $(u_k)_k$ by a reflection method and move the jump set inside $\Omega$. This will lead to a contradiction as we assumed that $(u_k)_k$ is a recovery sequence, but inside $\Omega$ the surface energy is much less than in $\Omega' \setminus {\Omega}$. In contrast to \cite{Giacomini-Ponsiglione:2006}, the construction is a bit more delicate here since the surface densities also depend on the crack opening.   Possibly after passing to a smaller $\rho$ (not relabeled), we can assume that in a suitable coordinate system
$$\Omega \cap Q_{4\rho} = \lbrace (x',y): \ x' \in (-4\rho,4\rho)^{d-1}, \ y \in (-4\rho,\tau(x')) \rbrace$$
for a Lipschitz function $\tau$ with $\Vert \tau \Vert_\infty \le \rho$. \UUU We choose $\eta \in (2\rho,3\rho)$ such that 
$$V_\rho := \lbrace (x',y): \ x' \in (-\rho,\rho)^{d-1}, \ y \in (\tau(x')-\eta,\tau(x')+\eta) \rbrace$$ satisfies $E'(u,\partial V_{\rho}) = 0 $. Note that $Q_\rho \subset V_\rho$ since $\eta > 2\rho$. \EEE Let $\hat{u}$ be the function defined on $V_{\rho}$ by reflecting $u$ at $\tau(x')$, $x' \in (-\rho,\rho)^{d-1}$, i.e.,
\begin{align*}
\hat{u}(x',y) = \begin{cases} 
u(x',y)  & y > \tau(x'),   \\  u(x', 2\tau(x') - y) &  y < \tau(x'). 
\end{cases}
\end{align*}
Clearly $\hat{u}\in W^{1,p}(V_{\rho};\R^m)$ as $u \in W^{1,p}(\Omega' \setminus \overline{\Omega};\R^m)$. In a similar fashion, we define $\hat{u}_k$ on $V_{\rho}$ by
\begin{align*}
\hat{u}_k(x',y) = \begin{cases} 
u_k(x',y)  & y > \tau(x') -\lambda_k,  \\  u_k(x', 2(\tau(x')-\lambda_k) - y) &  y < \tau(x')-\lambda_k, 
\end{cases}
\end{align*}
where $0 < \lambda_k \le 1/k$ is chosen such that
\begin{align}\label{eq: etk}
\mathcal{H}^{d-1} \Big(\Big\{ (x',y) \in J_{u_k}: \ x' \in (-\rho,\rho)^{d-1}, \ y \in (\tau(x')-\lambda_k,\tau(x')) \Big\}\Big) \le \frac{1}{k}.
\end{align} 
\UUU We note that  the functions are well defined since $Q_{4\rho} \subset \Omega'$, $\Vert \tau \Vert_\infty \le \rho$,  and $\eta < 3\rho$. \EEE We now introduce the sequence 
$$w_k:= u_k + \hat{u} - \hat{u}_k  \in   GSBV^p(V_\rho;\R^m).$$
The definition and $\lambda_k \to 0$ implies that  $w_k \to u$ in measure on $V_{\rho}$. Moreover, we find 
\begin{align}\label{eq: 3 prop}
(i)& \ \  \mathcal{H}^{d-1}( J_{w_k} \cap (V_\rho \setminus {\Omega}) ) = 0,\notag  \\ 
(ii)&  \ \  \UUU \mathcal{H}^{d-1}(J_{w_k} \setminus \Gamma_k)  \le\mathcal{H}^{d-1} \big(\big\{ (x',y) \in V_\rho \cap J_{u_k}: \ y > \tau(x')-\lambda_k \rbrace\big).  \EEE
\end{align}
Here, \BBB with the choice $\nu_{w_k} = \nu_{u_k}$ $\mathcal{H}^{d-1}$-a.e.\ on $J_{w_k} \cap J_{u_k}$, $\Gamma_k$ is defined by \EEE 
$$\Gamma_k: = \big\{ x \in J_{w_k} \cap J_{u_k}: \ [u_k](x) = [w_k](x)\big\}.$$
In particular, the jump of $w_k$ lies inside ${\Omega}$.  By ($g3$) and \eqref{eq: 3 prop}(i) we now find
$$G(w_k):=  \int_{J_{w_k} \cap V_\rho} g'_k(x, [w_k],\nu_{w_k}) \, d\mathcal{H}^{d-1} \le \int_{J_{u_k} \cap \Gamma_k} g'_k(x, [u_k],\nu_{u_k}) \, d\mathcal{H}^{d-1} + c_5 \mathcal{H}^{d-1}(J_{w_k} \setminus \Gamma_k).  $$
Then by \eqref{eq: etk}  and  \eqref{eq: 3 prop}(ii) we derive
\begin{align*}
G(w_k) \le \int_{J_{u_k} \cap \Gamma_k} g'_k(x, [u_k],\nu_{u_k}) \, d\mathcal{H}^{d-1} +c_5/k + c_5\mathcal{H}^{d-1}( J_{u_k} \cap (V_\rho \setminus {\Omega}) ).
\end{align*}
Therefore, by \eqref{eq: g ext}, \eqref{eq: sigma},  $\Gamma_k \subset J_{w_k} \subset \Omega \BBB \cap V_\rho \EEE$, \UUU and $Q_\rho \subset V_\rho$ \EEE we get 
\begin{align}\label{eq: to contra}
G(w_k) &\le \int_{J_{u_k} \cap \Gamma_k} g'_k(x, [u_k],\nu_{u_k}) \, d\mathcal{H}^{d-1} +c_5/k + (c_5+1)\mathcal{H}^{d-1}( J_{u_k} \cap (V_\rho \setminus {\Omega}) ) - \sigma \notag\\
& \le  \int_{J_{u_k} \cap V_\rho} g'_k(x, [u_k],\nu_{u_k}) \, d\mathcal{H}^{d-1} +c_5/k - \sigma.
\end{align}
On the other hand, recalling that $w_k \to u$ in measure on $V_\rho$, we have  by \eqref{eq: separation}  
$$ \int_{J_u \cap V_\rho} g'(x, [u],\nu_u) \, d\mathcal{H}^{d-1}  \le  \liminf_{k \to \infty} \int_{J_{w_k} \cap V_\rho} g'_k(x, [w_k],\nu_{w_k}) \, d\mathcal{H}^{d-1} = \liminf_{k \to \infty}  G(w_k).$$
Moreover, since $(u_k)_k$ is a recovery sequence for $u$ and $E'(u,\partial V_\rho) = 0$, \eqref{eq: separation-recovery} yields
$$ \int_{J_u \cap V_\rho} g'(x, [u],\nu_u) \, d\mathcal{H}^{d-1}  =  \lim_{k \to \infty} \int_{J_{u_k} \cap V_\rho} g'_k(x, [u_k],\nu_{u_k}) \, d\mathcal{H}^{d-1}.
$$
 The previous two equations contradict \eqref{eq: to contra}. This concludes the proof of (iii). 
\end{proof}

\begin{remark}\label{rem: bdy}
{\normalfont
Recalling the definition of the recovery sequence $\tilde{u}_k = u_k - \varphi_k$ below equation \eqref{eq: last g2}, we find $\mathcal{H}^{d-1}(J_{\tilde{u}_k} \setminus \Omega) \to 0$ by \eqref{eq: good approx}(iii) and  \eqref{eq: last g2}, i.e., except for an asymptotically vanishing part, the jump set is contained in $\Omega$. This shows that the surface density  $g'(x,\zeta,\nu)$ for $x \in \partial_D \Omega$ is completely determined by $(g_k)_k$, where $g_k: \Omega \times \R^m_0 \times \mathbb{S}^{d-1} \to \R_+$. \BBB In particular,  it is independent of the choice of $\Omega'$ and of the constant value $c'$ of
$g'_k$ on $\Omega' \setminus \Omega$ as long as $c'>c_5$. \EEE
}
\end{remark}
 
 \begin{remark}\label{rem: bdy2}
 {
 Consider the situation of Theorem  \ref{th: existence} for  $E \in \mathcal{E}'_{\Omega'}$ with densities $f,g$ such that $E(\cdot,A)$ is lower semicontinuous  in $L^0(A;\R^m)$ for all $A \in \mathcal{A}(\Omega')$. Consider the corresponding constant sequence $\tilde{E}'_k$  defined in Lemma \ref{lemma: gamma bdy} with densities \UUU given in \eqref{eq: f ext}-\eqref{eq: g ext}. \EEE Let $f',g'$ be the densities of the $\Gamma$-limit $\tilde{E}'$. One can show that $f(x,\xi) = f'(x,\xi)$  and $g(x,\zeta,\nu) = g'(x,\zeta,\nu)$ for $x\in\Omega$. The surface densities, however, may differ on $\partial_D \Omega$ since $g'$ is completely determined by the restriction of $g$ on $\Omega$ in the first variable, cf.\ Remark \ref{rem: bdy}. Consider, e.g., the densities $f(x,\xi) = c_1|\xi|^p$ and
 $$
 g(x,\zeta,\nu) = \begin{cases} c_5 &  x\in \Omega  \\c_4 & \in \Omega' \setminus \Omega, \end{cases}
 $$ 
 where  $c_4 < c_5$.  Then $g(x,\zeta,\nu) = c_4$ and $g'(x,\zeta,\nu) = c_5$ for $x \in \partial_D \Omega$.
 }
 \end{remark}

We close with a result about convergence of minimizers. 

\begin{theorem}[Convergence of minimizers]\label{th: Gamma existence}
Consider a  sequence of functionals $(\tilde{E}'_k)_k$ and the limiting energy $\tilde{E}'$ given by Lemma \ref{lemma: gamma bdy}, for boundary data $(h_k)_k \subset W^{1,p}(\Omega';\R^m)$ which converge strongly in $W^{1,p}(\Omega';\R^m)$ to $h$. Then 
\begin{align}\label{eq: eps control2}
\inf_{v \in L^0(\Omega';\R^m)} \tilde{E}'_k(v) \  \to \  \min_{v \in L^0(\Omega';\R^m)} \tilde{E}'(v)
\end{align}
 for $k \to \infty$. Moreover, for each sequence $(u_k)_k$ with
\begin{align}\label{eq: eps control}
\tilde{E}'_k(u_k) \le \inf_{v \in L^0(\Omega';\R^m)} \tilde{E}'_k(v) + \eps_k 
\end{align}
for some $\eps_k \to 0$, there exist a subsequence (not relabeled), modifications $(y_k)_k$ satisfying  $\mathcal{L}^d(\lbrace \nabla y_k \neq \nabla u_k \rbrace) \to 0$ as $k \to \infty$, and  $u \in GSBV^p(\Omega';\R^m)$ with  $y_k \to u$ in measure on $\Omega'$ such that 
$$\lim_{k \to \infty} \tilde{E}'_k(y_k) = \lim_{k \to \infty} \tilde{E}'_k(u_k) =  \tilde{E}'(u) = \min_{v \in L^0(\Omega';\R^m)} \tilde{E}'(v).$$
\end{theorem}
 
\begin{proof}
The statement follows in the spirit of  the fundamental theorem of $\Gamma$-convergence, see \cite[Theorem 1.21]{Braides:02}. 
Given $(u_k)_k \subset GSBV^p(\Omega';\R^m)$ satisfying \eqref{eq: eps control}, we apply Theorem \ref{th: comp} on the functionals $(E'_k)_k$ and find \BBB a subsequence (not relabeled), \EEE $(y_k)_k\subset GSBV^p(\Omega';\R^m)$ with $\mathcal{L}^d(\lbrace \nabla y_k \neq \nabla u_k \rbrace) \to 0$ and 
$$ 
 \liminf_{k \to \infty} \tilde{E}'_k(y_k) = \liminf_{k \to \infty} {E}'_k(y_k) \le \liminf_{k \to \infty} {E}'_k(u_k) = \liminf_{k \to \infty} \tilde{E}'_k(u_k) = \liminf_{k \to \infty}  \inf_{v \in L^0(\Omega';\R^m)} \tilde{E}'_k(v).$$
 Here, the first equality holds as $y_k = h_k$ on $\Omega' \setminus \overline{\Omega}$. By Theorem \ref{th: comp} we also get $u \in GSBV^p(\Omega';\R^m)$ satisfying $u = h$ on $\Omega' \setminus \overline{\Omega}$ with $y_k \to u$ in measure on $\Omega'$. Thus, by the $\Gamma$-liminf inequality in Lemma \ref{lemma: gamma bdy} we derive
 \begin{align}\label{eq: last1}
 \tilde{E}'(u) \le  \liminf_{k \to \infty} \tilde{E}'_k(y_k) \le  \liminf_{k \to \infty} \tilde{E}'_k(u_k) \le \liminf_{k \to \infty} \inf_{v \in L^0(\Omega';\R^m)} \tilde{E}'_k(v).\end{align}
 Again by Lemma \ref{lemma: gamma bdy}, for each $w \in L^0(\Omega';\R^m)$ we find a recovery sequence $(w_k)_k$ converging to $w$ in measure satisfying $\lim_{k\to \infty} \tilde{E}_k'(w_k) = \tilde{E}'(w)$.
 \BBB  This  implies
  \begin{align}\label{eq: last2}
\limsup_{k \to \infty}  \inf_{v \in L^0(\Omega';\R^m)} \tilde{E}'_k(v) \le \lim_{k \to \infty}  \tilde{E}'_k(w_k) =  \tilde{E}'(w).
  \end{align}
By combining \eqref{eq: last1}-\eqref{eq: last2} we find 
  \begin{align}\label{eq: last3}
 \tilde{E}'(u) \le \liminf_{k \to \infty} \inf_{v \in L^0(\Omega';\R^m)} \tilde{E}'_k(v) \le  \limsup_{k \to \infty} \inf_{v \in L^0(\Omega';\R^m)} \tilde{E}'_k(v) \le \tilde{E}'(w).
  \end{align}
  Since $w \in L^0(\Omega';\R^m)$ was arbitrary, we get that $u$ is a minimizer of $\tilde{E}'$. The statement follows from \eqref{eq: last1} and \eqref{eq: last3} with $w=u$. In particular, the limit in \eqref{eq: eps control2} does not depend on the specific choice of the subsequence and thus \eqref{eq: eps control2} holds for the whole sequence.  \EEE   
\end{proof}


\end{document}